\documentclass[review,hidelinks,onefignum,onetabnum]{siamart251216}
\newtheorem{example}{Example}[section]

\usepackage{lipsum}
\usepackage{amsfonts}
\usepackage{graphicx}
\usepackage{epstopdf}
\usepackage{algorithmic}
\ifpdf
  \DeclareGraphicsExtensions{.eps,.pdf,.png,.jpg}
\else
  \DeclareGraphicsExtensions{.eps}
\fi

\newsiamremark{remark}{Remark}
\newsiamremark{hypothesis}{Hypothesis}
\crefname{hypothesis}{Hypothesis}{Hypotheses}
\newsiamthm{claim}{Claim}

\headers{RPLSS: Randomized projected linear systems solver}{M.L. Xiao,  T. Li, and D. Needell}
\title{RPLSS: A randomized projected linear systems solver\thanks{Corresponding author: T. Li (tli@hainanu.edu.cn).
\funding{This work was funded by the National Natural Science Foundation of China [grant number 12401493].}}}
\author{Meng-Long Xiao\thanks{School of Mathematics and Statistics, Hainan University, Haikou 570228, P.R. China
(\email{tli@hainanu.edu.cn}, \email{xml@hainanu.edu.cn}).}
  \and Tao Li\footnotemark[2]
 \and Deanna Needell\thanks{Department of Mathematics, University of California, Los Angeles 90095,  USA  (\email{deanna@math.ucla.edu}).}}

\usepackage{amsopn}

\ifpdf
\hypersetup{
  pdftitle={RPLSS: A randomized projected linear systems solver},
  pdfauthor={Meng-Long Xiao,  Tao Li, Deanna Needell}
}
\fi

\begin{document}

\maketitle

\begin{abstract} The projected linear system solver (PLSS), by incrementally appending columns to a random or deterministic sketching matrix, provides an attractive finite termination property for consistent linear systems.
Nevertheless, a critical computational bottleneck of PLSS is accessing the whole coefficient matrix per iteration, making it prohibitive for extremely large-scale problems or applications with missing data. To alleviate this limitation, we propose a unified randomized PLSS (RPLSS) framework, built upon the tailored randomized
row or column selection strategies that require only partial matrix information per iteration, for
solving a general linear system, whether it is under- or overdetermined, and whether it is consistent
or not. Within this framework, we develop a randomized Gaussian Kaczmarz method and its extended variant as row-action solvers, and randomized coordinate descent variants as column-action solvers. Theoretically, we prove that our methods inherit the finite termination property of PLSS, while achieving an exponential convergence rate, overcoming the sluggish convergence inherent in conventional randomized Kaczmarz and coordinate descent methods. Numerical experiments demonstrate the superiority of our method against state-of-the-art randomized methods, particularly in scenarios with large missing data.



\end{abstract}


\section{Introduction}\label{section-1}
Consider a large, possibly consistent or not, general linear system
\begin{equation}\label{1-1}
\mathbf{A}\mathbf{x}=\mathbf{b},
\end{equation}
where $\mathbf{A}\in \mathbb{R}^{m\times n} $ and $\mathbf{b}\in \mathbb{R}^{m} $ are given, and $\mathbf{x}\in \mathbb{R}^{n} $ is unknown. Such a linear system underpins widespread applications across image reconstruction \cite{JA,GT}, signal processing \cite{Byr,Bor}, and partial differential equations \cite{Canu}, etc. While the regular iterative solvers, such as Krylov subspace methods and Hermitian and skew-Hermitian methods \cite{
HSU,HSU1,Bai0,CC,DC}, for \eqref{1-1} have been extensively matured, randomized solvers \cite{Strohmer,Needell,Elble,Needell2,KDu,Ma,Bai2,Zou,Needell3,Ke,Jiang} are growing in popularity due to a revolutionary advantage. Specifically, the traditional iterative solvers, demanding full-matrix access per iteration, are costly to generate highly accurate solutions, but they are not needed or even desired since Eq.\eqref{1-1} may constitute a random subset of a larger dataset, or it may be contaminated by noise. Conversely, randomized solvers address only a sequence of small projected systems, rather than \eqref{1-1}, while maintaining competitive convergence to yield the desired solutions. This makes them powerful for tackling large-scale, memory-constrained, or data-streaming applications. In this paper, we propose a family of randomized projection methods that solve a sequence of smaller systems and offer extraordinary advantages over state-of-the-art randomized solvers in terms of accuracy and convergence.

\subsection{Notations}\label{subsection-1-0}
Let $\mathbb{R}^{m\times n}$ be the collection of all real matrices with size $m\times n$. Denote by $\mathbf{A}^{(i)}$ and $\mathbf{A}_{(j)}$, the $i$-th row and
$j$-th column of a matrix $\mathbf{A}\in \mathbb{R}^{m\times n}$, respectively. For index sets $\tau_k$ and $\tilde{{\tau}}_k$, $\mathbf{A}_{\tau_k}$ and $\mathbf{A}_{:,\tilde{{\tau}}_k}$ represent the  row and column submatrices of $\mathbf{A}$, respectively. The symbol $\mathbf{e}_{j}$ represents the $j$-th column of the identity matrix $\mathbf{I}$, with dimension depending on the
context.
The Frobenius norm  of $\mathbf{A}$ is defined as
$
\|\mathbf{A}\|_F:= (\sum\limits_{j = 1}^n {\sum\limits_{i = 1}^m {{{\left| {{\mathbf{A}_{ij}}} \right|}^2}} } )^{\frac{1}{2}}.$
 Given a symmetric positive definite $\mathbf{B}\in \mathbb{R}^{n\times n}$, the scaled 2-norm of an $n$-vector $\mathbf{p}$ is defined as
$\|\mathbf{p}\|_\mathbf{B}=(\mathbf{p}^T\mathbf{B}\mathbf{p})^{\frac{1}{2}}$. When $\mathbf{B}=\mathbf{I}$, the norm reduces to the standard 2-norm. Denote by $\mathbf{A}^T$, $\mathbf{A}^{-1}$ and $\mathbf{A}^{\dagger}$, the transpose, inverse, and pseudo-inverse of $\mathbf{A}$, respectively, by $\mathcal{R}(\mathbf{A})$, the column space of $\mathbf{A}$, and by $\mathcal{R}(A)^{\bot}$, its orthogonal complement subspace. Let $\mathbf{b}_{\mathcal{R}(\mathbf{A})}$ and $\mathbf{b}_{\mathcal{R}(\mathbf{A})^{\bot}}$ be the orthogonal projection of $b$ onto $\mathcal{R}(\mathbf{A})$ and ${\mathcal{R}(\mathbf{A})^{\bot}}$, respectively. For the matrix $\mathbf{A}^T\mathbf{A}$, we denote its nonzero minimal eigenvalue by $\lambda_{min}(\mathbf{A}^T\mathbf{A})$, and the maximal eigenvalue by $\lambda_{max}(\mathbf{A}^T\mathbf{A})$. Denote by $\mathbb{E} _{k}$, the expected value conditional of the first $k$ iterations, i.e.,
$\mathbb{E} _{k} [\cdot ]=\mathbb{E}[\cdot \mid \tau_{0}, \tau_{1},\dots,\tau_{k-1}  ],$ where $\tau_{j} (j=0,1,\dots,k-1 )$ is the $j$-th index set chosen at the $j$-th iterate, and from the law of iterated expectations, we obtain $\mathbb{E}[\mathbb{E} _{k} [\cdot ]]=\mathbb{E}[\cdot]$. 

\subsection{Sketch-and-Project methods}\label{subsection-1-3}
Over the past decade, the sketch-and-project (SP) methods \cite{Woodruff,Ghashami,GoR} for solving the large-scale consistent linear systems, have witnessed significant attention due to their low memory and complexity. It compresses the original system into a simplified subproblem via sketching matrices per iteration, and then projects the current iterate onto the solution space of the sketched subproblem. Let $\mathbf{S}_k\in\mathbb{R}^{m\times r}$ be an arbitrary sketching matrix with $r \ll m$, subject to the specified dimensions and rank. Following the Johnson-Lindenstrauss lemma \cite{Johnson0}, a solution of the reduced linear system
\begin{equation}\label{1-11}
\mathbf{S}_k^T\mathbf{A}\mathbf{x}=\mathbf{S}_k^T\mathbf{b}
\end{equation}
is highly approximate $\mathbf{x}$ in \eqref{1-1}. Viewed from the angle of projection, the next iterate $\mathbf{x}_{k}$ of Eq.\eqref{1-11} yields \cite{RiP}
\begin{equation}\label{1-12}
\mathbf{x}_{k} = \arg\min_\mathbf{x}\frac{1}{2} \|\mathbf{x} - \mathbf{x}_{k-1}\|_\mathbf{B}, \quad \text{s.t.} \quad \mathbf{S}_k^T\mathbf{A}\mathbf{x}=\mathbf{S}_k^T\mathbf{b}.
\end{equation}
As was shown in \cite{GoR,Bjj}, if $\mathbf{S}_k$ is in the range of $\mathbf{A}$, then
\begin{equation}\label{1-13}
\mathbf{x}_{k}=\mathbf{x}_{k-1}+\mathbf{p}_{k},
\end{equation}
where the update \begin{equation}\label{a1-5}
\mathbf{p}_k = \mathbf{W} \mathbf{A}^T \mathbf{S}_k (\mathbf{S}_k^T\mathbf{A} \mathbf{W} \mathbf{A}^T\mathbf{S}_k)^{-1} \mathbf{S}_k^T \mathbf{r}_{k-1}  
\end{equation} with $\mathbf{W}=(\mathbf{B}^T\mathbf{B})^{-1}$; otherwise, the pseudo-inverse replaces the corresponding inverse. Note that there exist infinitely many iterative schemes like \eqref{1-13} since the matrices $\mathbf{S}_k$ and $\mathbf{B}$ in the update are arbitrary. In particular, when $\mathbf{B}=\mathbf{I}$ and $\mathbf{S}_k=\mathbf{e}_{i_k}$ (one random column of the identity), the iterative scheme \eqref{1-13} becomes  
\begin{equation}\label{1-14}
\mathbf{x}_{k} =\mathbf{x}_{k-1} + \frac{\mathbf{b}_{i_k}-\mathbf{A}^{(i_{k})}\mathbf{x}_{k-1}}{\left | \left | \mathbf{A}^{(i_{k})} \right |  \right |_2^{2} }(\mathbf{A}^{(i_{k})})^{T}.
\end{equation}
Obviously, they are the popular randomized
Kaczmarz methods \cite{Strohmer,Bai,Alderman}, the row-action solvers. When $\mathbf{B}=\mathbf{A}$ and $\mathbf{S}_k=\mathbf{A}\mathbf{e}_{j_{k}}$, i.e., $\mathbf{A}_{(j_{k})}$ (one random column of $\mathbf{A}$), then \eqref{1-13} is of the
form
\begin{equation}\label{1-15}
\mathbf{x}_{k} =\mathbf{x}_{k-1} + \frac{\mathbf{A}_{(j_{k})}^{T}(\mathbf{b}-\mathbf{A}\mathbf{x}_{k-1})}{\left | \left | \mathbf{A}_{(j_{k})} \right |  \right |_2^{2}}\mathbf{e}_{j_{k}},
\end{equation}
which are the well-known randomized coordinate descent or Gauss-Seidel methods \cite{LeD,Bai4,Ma}, the column-action solvers. Note that both methods employ simple sketching matrices to minimize per-iteration computational overhead, but their convergence depends heavily on the spectral properties of the matrix $\mathbf{A}$. When $\lambda_{min}(\mathbf{A}^T\mathbf{A})$ is sufficiently small, the convergence factor asymptotically nears one, resulting in prohibitively large iteration complexity and severely degraded convergence performance. 

To address the stagnant convergence inherent in traditional randomized methods, Brust and Saunders \cite{Bjj} proposed an elegant framework, called the projected linear systems solver (PLSS), for solving the consistent linear system \eqref{1-1}. Unlike the standard SP methods that independently regenerate the sketching matrix per iteration, the PLSS introduces an interesting process, that is, progressive sketching, which incrementally appends one column to the existing sketching matrix $\mathbf{S}_k = [\mathbf{s}_1, \mathbf{s}_2, \dots, \mathbf{s}_k]$. More precisely, by choosing the collection of all previous residuals as the sketching matrix, the PLSS decouples the cumbersome generalized inverse updates into a simple short recurrence formula, which only needs to store and update a few vectors, effectively controlling the memory overhead. Theoretically, the method converges to an exact solution within at most $\text{rank}(\mathbf{A})$ iterations in exact arithmetic.

It is worth pointing out that the core recurrence formulas and subspace orthogonality proofs of PLSS explicitly rely on the consistency assumption that $\mathbf{b}\in \mathcal{R}(\mathbf{A})$. When dealing with inconsistent systems arising from observational noise or modeling deficiencies, the sketching vectors, generated from residuals, often fail to accurately capture the true error projection of the solution space. This leads to numerical ill-conditioning or complete stagnation of the denominators in the short recurrence formula, making it unavailable. Moreover, the PLSS, maintaining the completeness of the residual sketch, strictly necessitates full access to the whole matrix $\mathbf{A}$ per iteration. This prevents the algorithm from leveraging local randomized sampling that can save operations. When encountering missing data, the PLSS does not work due to the impossibility of computing exact global residuals or executing full-matrix projections.

In this paper, we propose an efficient randomized framework, called Randomized Projected Linear Systems Solver (RPLSS), for solving Eq.\eqref{1-1}
whether it is under- or overdetermined, and whether it is consistent or not. The new framework lies in seamlessly integrating tailored randomized row or column selection strategies into the progressive sketching, which requires only partial matrix information per iteration. By utilizing localized sketching vectors that naturally accommodate missing or streaming data, the RPLSS maintains robust convergence and achieves substantial computational acceleration over state-of-the-art randomized solvers.
Moreover, the proposed solver harmonizes the local subspace finite termination property of the PLSS with the expected linear convergence of standard randomized methods. 

\subsection{ Our contributions}\label{subsection-1-4}
By incorporating the randomized
 Kaczmarz or coordinate descent methods into the RPLSS framework, we propose a family of randomized solvers for solving consistent or inconsistent linear systems. The contributions of this paper are as follows:

$\bullet$ An RPLSS with the Gaussian Kaczmarz (RPLSS-GK) method, by taking a random linear combination consisting of residuals and standard basis vectors, is developed for solving \eqref{1-1} when it is consistent. The RPLSS-GK
preserves the low cost while inheriting the finite-termination property of PLSS, demonstrating its robustness in row-missing scenarios.


$\bullet$ Drawing from the double sequence idea of the randomized extended Kaczmarz method, an extended version of RPLSS-GK is derived to solve \eqref{1-1} when it is inconsistent. The proposed method achieves expected linear convergence without requiring any row or column paving strategies.


$\bullet$ By randomly selecting column vectors as newly appended columns of the sketch, proportional to the square of the column norms, an RPLSS with the coordinate descent (RPLSS-CD) method is devised for solving \eqref{1-1} when it is inconsistent. Utilizing the orthogonal projection of PLSS to accumulate historical updates, the RPLSS-CD method effectively overcomes the slow convergence caused by single-component updates in traditional coordinate descent methods.


$\bullet$ Combining the current residual with randomly selected standard basis vectors flexibly, an RPLSS with the Gaussian coordinate descent (RPLSS-GCD) method is proposed. This method strikes an elegant balance between local coordinate updates and global residual feedback, which substantially accelerates its convergence on ill-conditioned and inconsistent problems.


The organization of this paper is as follows.  In Section \ref{section-2}, we first devise the RPLSS framework for solving Eq.\eqref{1-1}, and then propose two fast randomized row-action solvers, including RPLSS-GK and RPLSS-GEK, for the consistent and inconsistent cases, respectively.
 In Section \ref{section-3}, we propose two randomized column-action solvers. The convergence of all methods is also established. In Section \ref{section-4}, we provide abundant numerical results on large-scale sparse matrices, ill-conditioned systems, and applications with severe data missing, comparing the proposed algorithms against state-of-the-art methods to validate their superiority. Finally,  we conclude this paper by giving some remarks.

\section{RPLSS with Gaussian Kaczmarz methods}\label{section-2}
In this section, we seamlessly integrate randomized row- or column-selection strategies into the progressive sketching process, and then propose a randomized PLSS framework for the linear system \eqref{1-1}. Unlike the standard PLSS framework, the proposed framework only requires partial information of the coefficient matrix per iteration, which is particularly favorable for large-scale, memory-constrained, or streaming data problems. Moreover, by incrementally storing historical information in the sketching matrix, the  RPLSS also preserves the finite-termination property of PLSS. Indeed, by randomly constructing different sketch vectors $\mathbf{s}_k$, the various effective randomized methods can be developed for solving the consistent or inconsistent  Eq.\eqref{1-1}. 


As stated in \cite{Bjj}, the updates $\mathbf{p}_k=\mathbf{P}_{k-1}\mathbf{g}_{k-1}+\gamma_{k-1}\mathbf{y}_k$ are orthogonal, with $\mathbf{P}_k=[\mathbf{p}_1,\mathbf{p}_2,\dots,\mathbf{p}_k]$ and $\mathbf{y}_k=\mathbf{A}^T\mathbf{s}_k$. Following this
orthogonality, it is trivial to see that $$\mathbf{g}_{k-1}=-\gamma_{k-1}\mathbf{\Theta}_{k-1} ^{-1}\mathbf{P}_{k-1}^T\mathbf{y}_k,$$ where $\mathbf{\Theta}_{k-1}=\mathbf{P}_{k-1}^T\mathbf{P}_{k-1}$. Moreover, since $\mathbf{s}_k^T\mathbf{A}\mathbf{p}_k=\mathbf{s}_k^T\mathbf{r}_{k-1}$, it follows that $$\gamma_{k-1}=\frac{\mathbf{s}_k^T(\mathbf{b}-\mathbf{A}\mathbf{x}_{k-1})}{\|\mathbf{A}^T\mathbf{s}_k\|_2^2-\|\mathbf{d}_{k-1}\|_2^2},$$ where $\mathbf{d}_{k-1}=\mathbf{\Theta}_{k-1} ^{-1/2}\mathbf{P}_{k-1}^T\mathbf{A}^T\mathbf{s}_k$. Here, the $\mathbf{\Theta}_{k-1}^{-1/2}$denotes the inverse of the matrix square root of the diagonal matrix $\mathbf{\Theta}_{k-1}$. Therefore, the updates $\mathbf{p}_k$ are of the form

$$\mathbf{p}_k=\frac{\mathbf{s}_k^T(\mathbf{b}-\mathbf{A}\mathbf{x}_k)}{\|\mathbf{y}_k\|_2^2-\|\mathbf{d}_{k-1}\|_2^2}\mathbf{A}^T\mathbf{s}_k-\frac{\mathbf{s}_k^T(\mathbf{b}-\mathbf{A}\mathbf{x}_k)\mathbf{P}_{k-1}\mathbf{\Theta}_{k-1} ^{-1}\mathbf{P}_{k-1}^T}{\|\mathbf{y}_k\|_2^2-\|\mathbf{d}_{k-1}\|_2^2}\mathbf{A}^T\mathbf{s}_k.$$
If $\mathbf{W}\neq \mathbf{I}$, then we let $\mathbf{W}=\mathbf{R^T_{\mathbf{W}}}\mathbf{R_{\mathbf{W}}}$, that is, Cholesky decomposition. Define
$$\mathbf{x}_{k-1}=\mathbf{x}_{k-1}^\mathbf{W}\mathbf{R}_\mathbf{W}^{T},\quad\mathbf{p}_k=\mathbf{p}_k^\mathbf{W}\mathbf{R}_\mathbf{W}^{T},\quad\mathbf{A}_\mathbf{W} = \mathbf{A}\mathbf{R}_\mathbf{W}^T.$$
Then it follows that
$$\mathbf{p}_k^\mathbf{W} = \gamma_{k-1}^\mathbf{W} \Bigl( \mathbf{y}_k^\mathbf{W} - \mathbf{P}_{k-1}^\mathbf{W} \bigl((\mathbf{P}_{k-1}^{\mathbf{W}})^T\mathbf{P}_{k-1}^\mathbf{W}\bigr)^{-1} (\mathbf{P}_{k-1}^{\mathbf{W}})^T\mathbf{y}_k^\mathbf{W} \Bigr),$$
where $$\mathbf{y}_k^\mathbf{W} = \mathbf{A}_\mathbf{W}^T\mathbf{s}_{k},\,
\gamma_{k-1}^\mathbf{W} = \frac{\mathbf{s}_k^T(\mathbf{b}-\mathbf{A}\mathbf{x}_{k-1})}{\|\mathbf{y}_k^\mathbf{W}\|^2 - \|\mathbf{d}_{k-1}^\mathbf{W}\|^2},\,
\mathbf{d}_{k-1}^\mathbf{W} = \bigl((\mathbf{P}_{k-1}^{\mathbf{W}})^T\mathbf{P}_{k-1}^\mathbf{W}\bigr)^{-1/2}(\mathbf{P}_{k-1}^{\mathbf{W}})^T\mathbf{y}_k^\mathbf{W}.$$
From the above results, the update $\mathbf{p}_k$ can be rewritten as
$$\mathbf{p}_k=\gamma_{k-1}^{\mathbf{W}}(\mathbf{W}\mathbf{A}^T\mathbf{s}_k-\mathbf{P}_{k-1}(\mathbf{\Theta}_{k-1}^{\mathbf{W}}) ^{-1}\mathbf{P}_{k-1}^T\mathbf{A}^T\mathbf{s}_k),$$
where $\mathbf{d}_{k-1}^{\mathbf{W}}=\mathbf{\Theta}_{k-1} ^{-1/2}\mathbf{P}_{k-1}^T\mathbf{A}^T\mathbf{s}_k$ and $\mathbf{\Theta}_{k-1}^{\mathbf{W}}=\mathbf{P}_{k-1}^T\mathbf{W}^{-1}\mathbf{P}_{k-1}$. Remarkably, the PLSS with residual sketches requires the whole matrix-vector product $\mathbf{A}^T \mathbf{r}_{k}$ per iteration. When dealing with large-scale problems, this full-dimensional operation incurs high computational costs and slow convergence. 

By replacing the deterministic operation with a randomized subsampling of the rows or columns of $\mathbf{A}$, such as, the sketch vector $\mathbf{s}_k$, which consists of a random linear combination of the previous residual and the standard basis vectors, we introduce a randomized PLSS (RPLSS) framework for Eq.\eqref{1-1}. While sacrificing a marginal degree of precision, this framework significantly reduces the per-iteration cost, rendering it highly effective for large-scale applications or environments where the complete data is inaccessible. Now, we present the following RPLSS framework.
\begin{algorithm}
	\caption{RPLSS}
	\label{algo-1-1}
	\begin{algorithmic}[1]
        \REQUIRE $\mathbf{A}, \mathbf{b}$\\
        \ENSURE $\mathbf{x}_{k}$\\
            \STATE 
            if $\mathbf{B}^T\mathbf{B}\neq$ Empty, then \\
        \quad$\mathbf{W}=(\mathbf{B}^T\mathbf{B})^{-1}$
            \STATE 
            else \\
            \quad$\mathbf{W}=\mathbf{I}$
            \STATE
            end if
        \STATE
       For $k=0, 1, 2,\cdots $, do:\\
         \STATE
Randomly select a sketch vector $\mathbf{s}_k\in \mathbb{R}^{m}$
    \STATE
   Compute 
   $$\mathbf{p}_k=\gamma_{k-1}(\mathbf{W}\mathbf{A}^T\mathbf{s}_k-\mathbf{P}_{k-1}\mathbf{\Theta}_{k-1} ^{-1}\mathbf{P}_{k-1}^T\mathbf{A}^T\mathbf{s}_k)$$
   where $\gamma_{k-1}=\frac{\mathbf{s}_k^T(\mathbf{b}-\mathbf{A}\mathbf{x}_{k-1})}{\mathbf{s}_k^T\mathbf{A}\mathbf{W}\mathbf{A}^T\mathbf{s}_k-\|\mathbf{d}_{k-1}\|_2^2}$ with $\mathbf{d}_{k-1}=\mathbf{\Theta}_{k-1} ^{-1/2}\mathbf{P}_{k-1}^T\mathbf{A}^T\mathbf{s}_k$ and $\mathbf{\Theta}_{k-1}=\mathbf{P}_{k-1}^T\mathbf{W}^{-1}\mathbf{P}_{k-1}$.
            \STATE
            Set $\mathbf{x}_{k} =\mathbf{x}_{k-1}+\mathbf{p}_k$.

	\end{algorithmic}  
\end{algorithm}

As seen, the RPLSS framework gives rise to a family of different algorithms depending on the choice of the sketch vector $\mathbf{s}_k$. Here, we mainly consider two different the sketch vectors $\mathbf{s}_k=\sum\limits_{i\in \tau_k}\mathbf{r}^{(i)}_{k-1}\mathbf{e}_i$ or $\mathbf{s}_k=\sum\limits_{i\in\tau_k}\mathbf{A}\mathbf{e}_i$. The former is a  random linear
combination of the previous residual and the standard basis vectors, while the latter is a random submatrix that consists of some columns of $\mathbf{A}$. 

{From 
Theorem 3.1, Corollaries 3.2 and 3.3 presented in \cite{Bjj}, we make the following statement.
\begin{corollary}
\rm Let $\mathbf{x}_k$ be the iterates generated by Algorithm \ref{algo-1-1}, where at each step the sketching column is defined as $\mathbf{s}_k$(e.g. $\mathbf{s}_k = \sum_{i \in \tau_k} \mathbf{r}_{k-1}^{(i)} \mathbf{e}_i$), chosen so that the accumulated sketching matrix $\mathbf{S}_k = [\mathbf{s}_1,\dots,\mathbf{s}_k]$ maintains full column rank. Then there exists an integer $K \le \min(m,n)$ such that $\mathbf{x}_K = \mathbf{A}^\dagger \mathbf{b}$.
In other words, the algorithm converges to the least‑norm least‑squares solution in at most $\min(m,n)$ iterations.
\end{corollary}}

\subsection{RPLSS-GK}\label{subsection-2-1}
The randomized Kaczmarz-type methods, a leading class of randomized solvers, have received significant attention due to their simplicity and efficiency. For instance, Needell and Tropp \cite{Needell2} proposed a randomized block Kaczmarz method that greatly accelerates the convergence of the standard RK, but suffers from the high computational expense of evaluating pseudo-inverses.
Then, Necoara \cite{Necoara} pioneered the randomized average block Kaczmarz (RaBK) method without any pseudo-inverses. Moreover, Bai and Wu \cite{Bai} proposed a fast greedy randomized Kaczmarz (GRK) method, which adaptively selects rows based on residual magnitudes.

Overall, these randomized Kaczmarz solvers demonstrate that utilizing only a subset of row information can yield efficient convergence when solving large-scale and consistent linear systems. Motivated by this, we devise a similar strategy that randomly samples a portion of the residual vector to form the sketch vector $\mathbf{s}_k$, thereby circumventing the full-dimensional residual operations mandated by PLSS per iteration. Here, we set $\mathbf{s}_k=\sum\limits_{i\in \tau_k}\mathbf{r}^{(i)}_{k-1}\mathbf{e}_i$, a linear combination of previous residual and standard basis vectors, where $\tau_k$ is a randomly selected set of row indices. By integrating this strategy into the RPLSS framework, we propose an efficient RPLSS variant based on the Gaussian Kaczmarz method for solving the consistent linear system \eqref{1-1}.


\begin{algorithm}
	\caption{RPLSS-GK}
	\label{algo-2-1}
	\begin{algorithmic}[1]
        \REQUIRE $\mathbf{A}, \mathbf{b}$\\
        \ENSURE $\mathbf{x}_{k}$\\
            \STATE 
            if $\mathbf{B}^T\mathbf{B}\neq$ Empty, then \\
        \quad$\mathbf{W}=(\mathbf{B}^T\mathbf{B})^{-1}$
            \STATE 
            else \\
            \quad$\mathbf{W}=\mathbf{I}$
            \STATE
            end if
        \STATE
       For $k=0, 1, 2,\cdots $, do:\\
         \STATE
Generate an indicator set $\tau_{k}$, i.e., choosing $l m$ rows of $\mathbf{A}$ by using the simple random sampling, where $0<l<1$
 \STATE
   Compute 
   $$\mathbf{s}_k=\sum_{i\in \tau_k}\mathbf{r}^{(i)}_{k-1}\mathbf{e}_i$$
    \STATE
   Compute 
   $$\mathbf{p}_k=\gamma_{k-1}(\mathbf{W}\mathbf{A}^T\mathbf{s}_k-\mathbf{P}_{k-1}\mathbf{\Theta}_{k-1} ^{-1}\mathbf{P}_{k-1}^T\mathbf{A}^T\mathbf{s}_k)$$
   where $\gamma_{k-1}=\frac{\mathbf{s}_k^T(\mathbf{b}-\mathbf{A}\mathbf{x}_{k-1})}{\mathbf{s}_k^T\mathbf{A}\mathbf{W}\mathbf{A}^T\mathbf{s}_k-\|\mathbf{d}_{k-1}\|_2^2}$ with $\mathbf{d}_{k-1}=\mathbf{\Theta}_{k-1} ^{-1/2}\mathbf{P}_{k-1}^T\mathbf{A}^T\mathbf{s}_k$ and $\mathbf{\Theta}_{k-1}=\mathbf{P}_{k-1}^T\mathbf{W}^{-1}\mathbf{P}_{k-1}$.
            \STATE
            Set $\mathbf{x}_{k} =\mathbf{x}_{k-1}+\mathbf{p}_k$.

	\end{algorithmic}  
\end{algorithm}

\begin{rem}\label{Rem-1}
\rm Note that, when the parameter $l=1$, Algorithm \ref{algo-2-1} reduces to the PLSS with residual sketches method. For the general parameter $0<l<1$, Algorithm \ref{algo-2-1} also possesses the finite termination because it depends only on whether $\mathbf{S}_k$ is column-full rank, rather than the specific form of  $\mathbf{s}_k$. In fact, for step $5$ of the algorithm \ref{algo-2-1}, different rules for constructing the random indicator set will result in different RPLSS Gaussian Kaczmarz methods.
\end{rem}


We now state some necessary lemmas before proceeding to the convergence analysis.
\begin{lemma}\label{lem-1}
{\rm Let $\mathbf{x}_* = \mathbf{A}^\dagger\mathbf{b}$ be the least-norm solution of Eq.\eqref{1-1}, and $\overline{\mathbf{x}}_k=\mathbf{x}_k-\mathbf{x}_{\star}$ be the error vector at $k$-th iterate. Then the  matrix $\mathbf{P}_k=[\mathbf{p}_1,\mathbf{p}_2,\dots,\mathbf{p}_k]$ with orthogonal columns generated by Algorithm \ref{algo-2-1} yields $$\mathbf{P}_k^T \overline{\mathbf{x}}_k = \mathbf{0}.$$}
\end{lemma}
\begin{proof}
We prove the assertion by induction on the $k$-th iterate. For $k = 0$, it is trivial to see that the lemma is true.
Assume that $$\mathbf{P}_m^T \overline{\mathbf{x}}_m = \mathbf{0}$$ holds for $k = m$ ($m > 0$).
Then for $k = m+1$, it follows that
\begin{equation}\label{a2-1}
\mathbf{P}_{m+1}^T \overline{\mathbf{x}}_{m+1} = \begin{bmatrix} \mathbf{P}_m^T \\ \mathbf{p}_{m+1}^T \end{bmatrix} (\overline{\mathbf{x}}_m + \mathbf{p}_{m+1}) = \begin{bmatrix} \mathbf{P}_m^T \overline{\mathbf{x}}_m + \mathbf{P}_m^T \mathbf{p}_{m+1} \\ \mathbf{p}_{m+1}^T \overline{\mathbf{x}}_m + \|\mathbf{p}_{m+1}\|_2^2 \end{bmatrix}.  
\end{equation}
One can see from the assumption that the upper part of Eq.\eqref{a2-1} is a zero vector. For the lower part, it follows from Eq.\eqref{a1-5} that
$$
\begin{aligned}
\mathbf{p}_{m+1}^T \overline{\mathbf{x}}_m&= \mathbf{r}_m^T \mathbf{S}_{m+1} \left(\mathbf{S}_{m+1}^T \mathbf{A} \mathbf{A}^T \mathbf{S}_{m+1}\right)^{-1} \mathbf{S}_{m+1}^T \mathbf{A} \overline{\mathbf{x}}_m\\
&=-\mathbf{r}_m^T \mathbf{S}_{m+1} \left(\mathbf{S}_{m+1}^T \mathbf{A} \mathbf{A}^T \mathbf{S}_{m+1}\right)^{-1} \mathbf{S}_{m+1}^T \mathbf{A} \mathbf{r}_m
\end{aligned}
$$
and
$$
\begin{aligned}
\|\mathbf{p}_{m+1}\|_2^2 &= \mathbf{p}_{m+1}^T \mathbf{p}_{m+1} \\
&= \mathbf{r}_m^T \mathbf{S}_{m+1} \left(\mathbf{S}_{m+1}^T \mathbf{A} \mathbf{A}^T \mathbf{S}_{m+1}\right)^{-1}  \mathbf{S}_{m+1}^T \mathbf{r}_m.
\end{aligned}
$$
This shows that the lower part equals zero. Therefore, the assertion is surely valid.
\end{proof}

\begin{lemma}\cite{Necoara}\label{lem-2}
\rm Let $p_i = \mathbb{P}(i \in \tau_k)$ with $\tau_k$ being a random subset. Consider any deterministic sequence of scalars $\{\eta_i\}_{i=1}^m$, the following  holds:
$$\mathbb{E}_{\tau_k}\left[ \sum_{i \in \tau_k}  \eta _i \right] = \sum_{i=1}^m p_i  \eta _i.$$
\end{lemma}

\begin{lemma}\cite{Necoara}\label{lem-3}
\rm Consider without-replacement uniform sampling of a fixed block of size $\tau$ from the $m$ rows of $\mathbf{A}$. Then every row $i \in [m]$ has the same probability of being included in the random subset $\tau_k$, and this probability is given by
$$p_i= \mathbb{P}(i \in \tau_k)=\frac{\binom{m-1}{\tau-1}}{\binom{m}{\tau}} = \frac{\tau}{m},$$ 
where $\tau = |\tau_k|$ denotes the fixed cardinality of the sampled block.  
\end{lemma}

\begin{lemma}\cite{Sed}\label{lem-4}
\rm Let $X$ be a real‑valued random variable and $Y$ be a random variable such that $Y>0$ almost surely. Assume that $\mathbb{E}[X^2/Y]$ and $\mathbb{E}[Y]$ are finite. Then $$\mathbb{E} \left[ \frac{X^2}{Y} \right] \geq \frac{\left( \mathbb{E}[|X|] \right)^2}{\mathbb{E}[Y]} \geq \frac{\left( \mathbb{E}[X] \right)^2}{\mathbb{E}[Y]}.$$   
\end{lemma}

{\begin{lemma}\label{lem-5}
{\rm Let $\mathbf{x}_* = \mathbf{A}^\dagger\mathbf{b}$ be the least-norm solution of Eq.\eqref{1-1}, and $\overline{\mathbf{x}}_k=\mathbf{x}_k-\mathbf{x}_{\star}$ be the error vector at $k$-th iterate. Then
$$1-\frac{l\lambda_{min}^2(\mathbf{A}^T\mathbf{A})}{\sin^2(\theta_{k-1})\lambda_{max}^2(\mathbf{A}^T\mathbf{A})}\geq0,
$$
where $\theta_{k-1}$ is the angle between the vector $\mathbf{A}^T\mathbf{s}_k$ and the column space of $\mathbf{P}_{k-1}$.
}
\end{lemma}
\begin{proof}
Define $\mathbf{y}_k=\mathbf{A}^T\mathbf{s}_k$ and $\mathbf{Q}_{k-1}=\mathbf{P}_{k-1}\mathbf{\Theta}_{k-1} ^{-1}\mathbf{P}_{k-1}^T$. By taking the conditional expectation of block $\tau_k$ over the entire history $\mathcal{F}_{k-1} = \{\tau_1,\cdots, \tau_{k-1}\}$, it follows that
$$
\begin{aligned}
\mathbb{E}_{\tau_k}[|\overline{\mathbf{x}}_{k-1}^T\mathbf{y}_k|^2|\mathcal{F}_{k-1}]&=(\mathbb{E}_{\tau_k}[\|\mathbf{A}_{\tau{_k}}\overline{\mathbf{x}}_{k-1}\|_2^2|\mathcal{F}_{k-1}])^2\\
&=(l\|\mathbf{A}\overline{\mathbf{x}}_{k-1}\|_2^2)^2\\
&\geq l^2\lambda_{min}^2(\mathbf{A}^T\mathbf{A})\|\overline{\mathbf{x}}_{k-1}\|_2^4.
\end{aligned}
$$
Form Lemma \ref{lem-1}, we have 
$$
\begin{aligned}
|\overline{\mathbf{x}}_{k-1}^T\mathbf{y}_k|^2&=|\overline{\mathbf{x}}_{k-1}^T\mathbf{Q}_{k-1}\mathbf{y}_k+\overline{\mathbf{x}}_{k-1}^T(I-\mathbf{Q}_{k-1})\mathbf{y}_k|\\
&=|\overline{\mathbf{x}}_{k-1}^T(I-\mathbf{Q}_{k-1})\mathbf{y}_k|^2\\
&\leq\|\overline{\mathbf{x}}_{k-1}\|_2^2\|(I-\mathbf{Q}_{k-1})\mathbf{y}_k\|_2^2\\
&\leq\|\overline{\mathbf{x}}_{k-1}\|_2^2\sin^2(\theta_{k-1})\|\mathbf{y}_k\|_2^2.\\
\end{aligned}
$$
Then 
\begin{equation}\label{22-2}
\begin{aligned}
\mathbb{E}_{\tau_k}[\|\mathbf{y}_k\|_2^2|\mathcal{F}_{k-1}] &=\mathbb{E}_{\tau_k}[\|\mathbf{A}_{\tau_k}^T\mathbf{A}_{\tau_k}\overline{\mathbf{x}}_{k-1}\|_2^2|\mathcal{F}_{k-1}]\\
&= \mathbb{E}_{\tau_k}\left[ \left\| \sum_{i \in \tau_k} (\mathbf{A}^{(i)} \overline{\mathbf{x}}_{k-1}) (\mathbf{A}^{(i)})^T \right\|_2^2 \,\bigg|\, \mathcal{F}_{k-1} \right] \\
&= p_{ij}(i \in \tau_k, j \in \tau_k)\sum_{i=1}^m \sum_{j=1}^m (\mathbf{A}^{(i)} \overline{\mathbf{x}}_{k-1})(\mathbf{A}^{(j)} \overline{\mathbf{x}}_{k-1}) \mathbf{A}^{(i)} (\mathbf{A}^{(j)})^T  \\
&=p_{i}(i \in \tau_k)\sum_{i=1}^m(\mathbf{A}^{(i)} \overline{\mathbf{x}}_{k-1})^2\|\mathbf{A}^{(i)}\|_2^2 \\
&\quad+ p_{ij}(i,j \in \tau_k)\sum_{i\neq j}(\mathbf{A}^{(i)} \overline{\mathbf{x}}_{k-1})(\mathbf{A}^{(j)} \overline{\mathbf{x}}_{k-1}) \mathbf{A}^{(i)} (\mathbf{A}^{(j)})^T\\
&\leq l^2\|\mathbf{A}^T\mathbf{A}\overline{\mathbf{x}}_{k-1}\|_2^2+l(1-l)\lambda_{max}(\mathbf{A}^T\mathbf{A})\|\mathbf{A}\overline{\mathbf{x}}_{k-1}\|_2^2\\
&\leq l\lambda_{max}^2(\mathbf{A}^T\mathbf{A})\|\overline{\mathbf{x}}_{k-1}\|_2^2.
\end{aligned}
\end{equation}
From the above results, it follows that 
$$l^2\lambda_{min}^2(\mathbf{A}^T\mathbf{A})\|\overline{\mathbf{x}}_{k-1}\|_2^4\leq\|\overline{\mathbf{x}}_{k-1}\|_2^2\sin^2(\theta_{k-1})l\lambda_{max}^2(\mathbf{A}^T\mathbf{A})\|\overline{\mathbf{x}}_{k-1}\|_2^2.$$
If $\|\overline{\mathbf{x}}_{k-1}\|_2=0$, then $\mathbf{x}_{k-1}$ is an exact solution of Eq.\eqref{1-1}; otherwise,
$$\sin^2(\theta_{k-1})\geq\frac{l\lambda_{min}^2(\mathbf{A}^T\mathbf{A})}{\lambda_{max}^2(\mathbf{A}^T\mathbf{A})}>0.$$
This completes the proof.
\end{proof}}

From the above lemmas, we will establish the following results regarding the convergence rate of Algorithm \ref{algo-2-1}. To simplify the analysis, we let $\mathbf{W}=\mathbf{I}$.

\begin{theorem}\label{The-2}
{\rm Let $\mathbf{x}_{\star}=\mathbf{A}^{\dagger}\mathbf{b}$ be the least-norm solution of consistent linear systems \eqref{1-1}. Then the iterative sequence $\{
\mathbf{x}_{k}\}$ generated by Algorithm \ref{algo-2-1} converges to $\mathbf{x}_{\star }$ for any initial vector $\mathbf{x}_0\in \mathcal{R}(\mathbf{A}^T)$. Moreover, the corresponding error norm in expectation yields
}
\begin{equation}\label{2-2}
\mathbb{E}_{k}\|\mathbf{x}_{k}-\mathbf{x}_{\star}\|_2^2\leq\left(1-\frac{l\lambda_{min}^2(\mathbf{A}^T\mathbf{A})}{\sin^2(\theta_{k-1})\lambda_{max}^2(\mathbf{A}^T\mathbf{A})}\right)\|\mathbf{x}_{k-1}-\mathbf{x}_{\star}\|_2^2.
\end{equation}
\end{theorem}
\begin{proof}
Define
$$\begin{aligned}\Delta _{k} &=(I-\mathbf{P}_{k-1}\mathbf{\Theta}_{k-1} ^{-1}\mathbf{P}_{k-1}^T)\mathbf{A}^T\mathbf{s}_k\mathbf{s}_k^T\mathbf{A},\\
\phi_k&=\|\mathbf{A}^T\mathbf{s}_k\|_2^2-\|\mathbf{d}_{k-1}\|_2^2.\end{aligned}$$
It follows from Algorithm \ref{algo-2-1} that
$$
\begin{aligned}
\mathbf{x}_{k}-\mathbf{x}_{\star}&=\mathbf{x}_{k-1}-\mathbf{x}_{\star}+\frac{\mathbf{s}_k^T(\mathbf{b}-\mathbf{A}\mathbf{x}_{k-1})-\mathbf{s}_k^T(\mathbf{b}-\mathbf{A}\mathbf{x}_{k-1})\mathbf{P}_{k-1}\mathbf{\Theta}_{k-1} ^{-1}\mathbf{P}_{k-1}^T}{\|\mathbf{y}_k\|_2^2-\|\mathbf{d}_{k-1}\|_2^2}\mathbf{A}^T\mathbf{s}_k\\
&=(\mathbf{I}-P)(\mathbf{x}_{k-1}-\mathbf{x}_{\star}),
\end{aligned}
$$
where $P=\frac{\mathbf{A}^T\mathbf{s}_k\mathbf{s}_k^T\mathbf{A}-\mathbf{P}_{k-1}\mathbf{\Theta}_{k-1} ^{-1}\mathbf{P}_{k-1}^T\mathbf{A}^T\mathbf{s}_k\mathbf{s}_k^T\mathbf{A}}{\|\mathbf{A}^T\mathbf{s}_k\|_2^2-\|\mathbf{d}_{k-1}\|_2^2}$. Obviously, $(\mathbf{I}-P)^2=\mathbf{I}-P$.
Therefore,
$$
\begin{aligned}
\|\mathbf{x}_{k}-\mathbf{x}_{\star}\|_2^2&=\|(\mathbf{I}-P)(\mathbf{x}_{k-1}-\mathbf{x}_{\star})\|_2^2\\
&=\|\mathbf{x}_{k-1}-\mathbf{x}_{\star}\|_2^2-\|P(\mathbf{x}_{k-1}-\mathbf{x}_{\star})\|_2^2\\
&=\|\mathbf{x}_{k-1}-\mathbf{x}_{\star}\|_2^2-\|\frac{(\mathbf{I}-\mathbf{P}_{k-1}\mathbf{\Theta}_{k-1} ^{-1}\mathbf{P}_{k-1}^T)\mathbf{A}^T\mathbf{s}_k\mathbf{s}_k^T\mathbf{A}}{\|\mathbf{A}^T\mathbf{s}_k\|_2^2-\|\mathbf{d}_{k-1}\|_2^2}(\mathbf{x}_{k-1}-\mathbf{x}_{\star})\|_2^2\\
&=\|\mathbf{x}_{k-1}-\mathbf{x}_{\star}\|_2^2-\frac{\|\Delta_k(\mathbf{x}_{k-1}-\mathbf{x}_{\star})\|_2^2}{\phi_k^2}\\
&=\|\mathbf{x}_{k-1}-\mathbf{x}_{\star}\|_2^2-\frac{(\|\mathbf{A}^T\mathbf{s}_k\|_2^2-\|\mathbf{d}_{k-1}\|_2^2)|\mathbf{s}_k^T\mathbf{A}(\mathbf{x}_{k-1}-\mathbf{x}_{\star})|^2}{\phi_k^2}\\
&=\|\mathbf{x}_{k-1}-\mathbf{x}_{\star}\|_2^2-\frac{|\mathbf{s}_k^T\mathbf{A}(\mathbf{x}_{k-1}-\mathbf{x}_{\star})|^2}{\phi_k}\\
&=\|\mathbf{x}_{k-1}-\mathbf{x}_{\star}\|_2^2-\frac{|\mathbf{s}_k^T(\mathbf{b}-\mathbf{A}\mathbf{x}_{k-1})|^2}{\phi_k}\\
&=\|\mathbf{x}_{k-1}-\mathbf{x}_{\star}\|_2^2-\frac{|\mathbf{s}_k^T(\mathbf{b}-\mathbf{A}\mathbf{x}_{k-1})|^2}{\sin^2(\theta_{k-1})\|\mathbf{A}^T\mathbf{s}_k\|_2^2}\\
&=\|\mathbf{x}_{k-1}-\mathbf{x}_{\star}\|_2^2-\frac{\|\mathbf{b}_{\tau_k}-\mathbf{A}_{\tau_k}\mathbf{x}_{k-1}\|_2^4}{\sin^2(\theta_{k-1})\|\mathbf{A}^T\mathbf{s}_k\|_2^2},\\
\end{aligned}
$$
where $\theta_{k-1}$ is similarly defined as  in Lemma \ref{lem-5}.
The result follows from taking the conditional expectation of block $\tau_k$  over the entire history  $\mathcal{F}_{k-1} = \{\tau_1,\cdots, \tau_{k-1}\}$, which yields
$$
\begin{aligned}
\mathbb{E}_{\tau_k}[\|\mathbf{x}_{k}-\mathbf{x}_{\star}\|_2^2|\mathcal{F}_{k-1}]&\leq \|\mathbf{x}_{k-1}-\mathbf{x}_{\star}\|_2^2-\mathbb{E}_{\tau_k}\left[\frac{\|\mathbf{b}_{\tau_k}-\mathbf{A}_{\tau_k}\mathbf{x}_{k-1}\|_2^4}{\sin^2(\theta_{k-1})\|\mathbf{A}^T\mathbf{s}_k\|_2^2}\Bigg| \mathcal{F}_{k-1}\right]\\
&\leq\|\mathbf{x}_{k-1}-\mathbf{x}_{\star}\|_2^2-\frac{\mathbb{E}_{\tau_k}[\|\mathbf{b}_{\tau_k}-\mathbf{A}_{\tau_k}\mathbf{x}_{k-1}\|_2^4|\mathcal{F}_{k-1}]}{\sin^2(\theta_{k-1})\mathbb{E}_{\tau_k}[\|\mathbf{A}^T\mathbf{s}_k\|_2^2|\mathcal{F}_{k-1}]}\\
&=\|\mathbf{x}_{k-1}-\mathbf{x}_{\star}\|_2^2-\frac{(\sum_{i=1}^m p_i|r^{(i)}_{k-1}|^2)^2}{\sin^2(\theta_{k-1})\mathbb{E}_{\tau_k}[\|\mathbf{A}^T\mathbf{s}_k\|_2^2|\mathcal{F}_{k-1}]}\\
&=\|\mathbf{x}_{k-1}-\mathbf{x}_{\star}\|_2^2-\frac{l^2\|\mathbf{A}(\mathbf{x}_{k-1}-\mathbf{x}_{\star})\|_2^4}{\sin^2(\theta_{k-1})\mathbb{E}_{\tau_k}[\|\mathbf{A}^T\mathbf{s}_k\|_2^2|\mathcal{F}_{k-1}]}\\
&\leq\|\mathbf{x}_{k-1}-\mathbf{x}_{\star}\|_2^2-\frac{l^2\|\mathbf{A}(\mathbf{x}_{k-1}-\mathbf{x}_{\star})\|_2^4}{\sin^2(\theta_{k-1})l\lambda_{max}^2(\mathbf{A}^T\mathbf{A})\|\mathbf{x}_{k-1}-\mathbf{x}_{\star}\|_2^2}\\
&\leq(1-\frac{l\lambda_{min}^2(\mathbf{A}^T\mathbf{A})}{\sin^2(\theta_{k-1})\lambda_{max}^2(\mathbf{A}^T\mathbf{A})})\|\mathbf{x}_{k-1}-\mathbf{x}_{\star}\|_2^2.
\end{aligned}
$$
Consequently, the assertion follows immediately by taking the expectation over the entire history.
\end{proof}


\subsection{RPLSS-GEK}\label{subsection-2-2}
When the linear system \eqref{1-1} is inconsistent, Needell \cite{Needell} showed that RK fails to converge to the least-norm least-squares solution. To solve this problem, Zouzias and Freris \cite{Zou} first proposed a randomized extended Kaczmarz (REK) method. More concisely, the REK solver iteratively updates two vectors $\mathbf{z}_k$ and $\mathbf{x}_k$ ($\mathbf{A}^*\mathbf{z}=0$ for $\mathbf{z}_k$ and $\mathbf{A}\mathbf{x}=\mathbf{b}-\mathbf{z}_k$ for $\mathbf{x}_k$) by
$$
\mathbf{z}_{k+1} =(I_m-\frac{\mathbf{A}_{(j_k)}\mathbf{A}_{(j_k)}^T}{\|\mathbf{A}_{(j_k)}\|_2^2})\mathbf{z}_{k},
$$
$$
\mathbf{x}_{k+1} =\mathbf{x}_k+\frac{{\mathbf{b}}^{(i_{k})}-\mathbf{z}_k^{(i_k)}-\mathbf{A}^{(i_{k})}\mathbf{x}_{k}} {\| \mathbf{A}^{(i_{k})}\|_{2} ^{2} }(\mathbf{A}^{(i_{k})})^{T}.
$$
To further improve its convergence, a block version of REK has been proposed in \cite{Needell3}. Moreover, Du et al. devised a novel randomized extended average block Kaczmarz method \cite{KDu} for finding the least-squares solution of Eq.\eqref{1-1} without involving the pseudo-inverse. Ke et al. \cite{Ke} derived a greedy block extended Kaczmarz (GBEK) method with the average block projection technique. Besides, Bai and Wu \cite{Bai2} further extended the greedy randomized Kaczmarz framework to inconsistent systems.

From these randomized extended Kaczmarz solvers, we see that double‑sequence iteration is capable of circumventing the convergence failure of single‑sequence iteration on large‑scale inconsistent linear systems.
Moreover, when the linear system \eqref{1-1} is inconsistent, PLSS also fails to converge to the least-squares solution. To address this drawback, we develop an extended version of the RPLSS Gaussian Kaczmarz method, which addresses the slow convergence issue of conventional extended Kaczmarz methods. Specifically, referencing the double sequence iteration of REK, we maintain two sequences: one for $\mathbf{z}_k$ that drives $\mathbf{A}^T\mathbf{z}$ toward zero, and one for $\mathbf{x}_k$ that solves $\mathbf{A}\mathbf{x} = \mathbf{b} - \mathbf{z}_k$.

In the RPLSS framework, the update for $\mathbf{z}_k$ is obtained by solving the projected systems
\begin{equation}\label{2-3}
\bar{\mathbf{S}}_k^T\mathbf{A}^T\mathbf{z}_k=0  ,
\end{equation}
where $\bar{\mathbf{S}}_k$ is a sketch matrix. This yields the update $\mathbf{z}_k=\mathbf{z}_{k-1}-\bar{\mathbf{p}}_k$ with $$\bar{\mathbf{p}}_k=\bar{\gamma}_{k-1}(\mathbf{A}\bar{\mathbf{s}}_k-\bar{\mathbf{P}}_{k-1}\bar{\mathbf{\Theta}}_{k-1} ^{-1}\bar{\mathbf{P}}_{k-1}^T\mathbf{A}\bar{\mathbf{s}}_k)$$
   where $\bar{\gamma}_{k-1}=\frac{\bar{\mathbf{s}}_k^T\mathbf{A}^T\mathbf{z}_{k-1}}{\|\mathbf{A}\bar{\mathbf{s}}_k\|_2^2-\|\bar{\mathbf{d}}_{k-1}\|_2^2}$ with $\bar{\mathbf{d}}_{k-1}=\bar{\mathbf{\Theta}}_{k-1} ^{-1/2}\bar{\mathbf{P}}_{k-1}^T\mathbf{A}\bar{\mathbf{s}}_k$. The detail of $\bar{\mathbf{p}}_k$ is provided in Appendix \ref{App-1}. Similarly, the update for $\mathbf{x}_k$ follows the RPLSS‑GK solver, except that the right side is replaced with $\mathbf{b} - \mathbf{z}_k$ and the history of previous updates is no longer used, i.e., $\mathbf{P}_{k-1}=\mathbf{0}$.
   
From what was said above, we now present the RPLSS Gaussian extended Kaczmarz (RPLSS-GEK) method as follows.
\begin{algorithm}
	\caption{RPLSS-GEK}
	\label{algo-2-2}
	\begin{algorithmic}[1]
        \REQUIRE $\mathbf{A},\mathbf{b},\mathbf{z}_0= \mathbf{b}$\\
        \ENSURE $\mathbf{x}_{k}$\\
        \STATE
       For $k= 1, 2,\cdots $, until convergence, do:\\
                \STATE
Generate an indicator set $\tilde{{\tau}}_k$, i.e., choosing $l n$ columns of $\mathbf{A}$ by using the simple random sampling, where $0<l<1$
 \STATE
   Compute 
   $$\bar{\mathbf{s}}_k=\sum_{j\in \tilde{{\tau}}_k}\mathbf{A}_{(j)}^T\mathbf{z}_{k-1}\mathbf{e}_j$$
   Compute 
   $$\bar{\mathbf{p}}_k=\bar{\gamma}_{k-1}(\mathbf{A}\bar{\mathbf{s}}_k-\bar{\mathbf{P}}_{k-1}\bar{\mathbf{\Theta}}_{k-1} ^{-1}\bar{\mathbf{P}}_{k-1}^T\mathbf{A}\bar{\mathbf{s}}_k)$$
   where $\bar{\gamma}_{k-1}=\frac{\bar{\mathbf{s}}_k^T\mathbf{A}^T\mathbf{z}_{k-1}}{\|\mathbf{A}\bar{\mathbf{s}}_k\|_2^2-\|\bar{\mathbf{d}}_{k-1}\|_2^2}$ with $\bar{\mathbf{d}}_{k-1}=\bar{\mathbf{\Theta}}_{k-1} ^{-1/2}\bar{\mathbf{P}}_{k-1}^T\mathbf{A}\bar{\mathbf{s}}_k$
            \STATE
            Set $\mathbf{z}_{k} =\mathbf{z}_{k-1}-\bar{\mathbf{p}}_k$.

         \STATE
Generate an indicator set $\tau_{k}$, i.e., choosing $l m$ rows of $\mathbf{A}$ by using the simple random sampling, where $0<l<1$
 \STATE
   Compute 
   $$\mathbf{s}_k=\sum_{i\in \tau_k}\mathbf{r}^{(i)}_{k-1}\mathbf{e}_i,$$
      where $\mathbf{r}_{k-1}=\mathbf{b}-\mathbf{z}_{k-1}-\mathbf{A}\mathbf{x}_{k-1}$.
    \STATE
   Compute 
   $$\mathbf{p}_k=\gamma_{k-1}\mathbf{A}^T\mathbf{s}_k$$
   where $\gamma_{k-1}=\frac{\mathbf{s}_k^T(\mathbf{b}-\mathbf{z}_{k-1}-\mathbf{A}\mathbf{x}_{k-1})}{\|\mathbf{A}^T\mathbf{s}_k\|_2^2}$. 
            \STATE
            Set $\mathbf{x}_{k} =\mathbf{x}_{k-1}+\mathbf{p}_k$.

	\end{algorithmic}  
\end{algorithm}

An important result for guaranteeing the convergence of Algorithm \ref{algo-2-2}  is stated in the following theorem.
\begin{theorem}\label{The-4}
{\rm The iterative sequence $\{\mathbf{z}_{k}\}$ generated by Algorithm \ref{algo-2-2} converges to $\mathbf{b}_{\mathcal{R}(\mathbf{A})^{\bot}}$. Moreover, the corresponding error norm in expectation yields }
\end{theorem}
\begin{equation}\label{2-4}
\mathbb{E}\|\mathbf{z}_k-\mathbf{b}_{\mathcal{R}(\mathbf{A})^{\bot}}\|_2^2\leq\hat{\beta} ^k\|\mathbf{b}_{\mathcal{R}(\mathbf{A})}\|_2^2,   
\end{equation}
where $\hat{\beta}=\sup\limits_k\left(1-\frac{l\lambda_{min}^2(\mathbf{A}\mathbf{A}^T)}{\sin^2(\tilde{\theta}_{k-1})\lambda_{max}^2(\mathbf{A}\mathbf{A}^T)}\right)$.
\begin{proof}
Let $P=\frac{\mathbf{A}\bar{\mathbf{s}}_k\bar{\mathbf{s}}_k^T\mathbf{A}^T-\bar{\mathbf{P}}_{k-1}\bar{\mathbf{\Theta}}_{k-1} ^{-1}\bar{\mathbf{P}}_{k-1}^T\mathbf{A}\bar{\mathbf{s}}_k\bar{\mathbf{s}}_k^T\mathbf{A}^T}{\|\mathbf{A}\bar{\mathbf{s}}_k\|_2^2-\|\bar{\mathbf{d}}_{k-1}\|_2^2}$. Trivially, $(\mathbf{I}-P)^2=\mathbf{I}-P$. Define $\overline{\mathbf{e}}_k:=\mathbf{z}_k-\mathbf{b}_{\mathcal{R}(\mathbf{A})^{\bot}}$ for every $k \geq 0$. From Algorithm \ref{algo-2-2} and $(\mathbf{I}-P)\mathbf{b}_{\mathcal{R}(A)^{\bot}}=\mathbf{b}_{\mathcal{R}(A)^{\bot}}$, it can be stated that
\begin{equation}\label{2-5}
\begin{aligned}
\overline{\mathbf{e}}_k&=\overline{\mathbf{e}}_{k-1}-\frac{\bar{\mathbf{s}}_k^T\mathbf{A}^T\mathbf{z}_{k-1}-\bar{\mathbf{s}}_k^T\mathbf{A}^T\mathbf{z}_{k-1}\bar{\mathbf{P}}_{k-1}\bar{\mathbf{\Theta}}_{k-1} ^{-1}\bar{\mathbf{P}}_{k-1}^T}{\|\mathbf{A}\bar{\mathbf{s}}_k\|_2^2-\|\bar{\mathbf{d}}_{k-1}\|_2^2}\mathbf{A}\bar{\mathbf{s}}_k\\
&=(\mathbf{I}-P)\overline{\mathbf{e}}_{k-1}.
\end{aligned}
\end{equation}
 It follows from Eq.\eqref{2-5} that
$$
\begin{aligned}
\|\overline{\mathbf{e}}_k\|_2^2&=\|(\mathbf{I}-P)\overline{\mathbf{e}}_{k-1}\|_2^2\\
&=\|\overline{\mathbf{e}}_{k-1}\|_2^2-\left\|\frac{(\mathbf{I}-\bar{\mathbf{P}}_{k-1}\bar{\mathbf{\Theta}}_{k-1} ^{-1}\bar{\mathbf{P}}_{k-1}^T)\mathbf{A}\bar{\mathbf{s}}_k\bar{\mathbf{s}}_k^T\mathbf{A}^T}{\|\mathbf{A}\bar{\mathbf{s}}_k\|_2^2-\|\bar{\mathbf{d}}_{k-1}\|_2^2}\overline{\mathbf{e}}_{k-1}\right\|_2^2\\
&=\|\overline{\mathbf{e}}_{k-1}\|_2^2-\frac{\|(\mathbf{I}-\bar{\mathbf{P}}_{k-1}\bar{\mathbf{\Theta}}_{k-1} ^{-1}\bar{\mathbf{P}}_{k-1}^T)\mathbf{A}\bar{\mathbf{s}}_k\bar{\mathbf{s}}_k^T\mathbf{A}^T\overline{\mathbf{e}}_{k-1}\|_2^2}{(\|\mathbf{A}\bar{\mathbf{s}}_k\|_2^2-\|\bar{\mathbf{d}}_{k-1}\|_2^2)^2}\\
&=\|\overline{\mathbf{e}}_{k-1}\|_2^2-\frac{(\|\mathbf{A}\bar{\mathbf{s}}_k\|_2^2-\|\bar{\mathbf{d}}_{k-1}\|_2^2)|\bar{\mathbf{s}}_k^T\mathbf{A}^T\overline{\mathbf{e}}_{k-1}|^2}{(\|\mathbf{A}\bar{\mathbf{s}}_k\|_2^2-\|\bar{\mathbf{d}}_{k-1}\|_2^2)^2}\\
&=\|\overline{\mathbf{e}}_{k-1}\|_2^2-\frac{|\bar{\mathbf{s}}_k^T\mathbf{A}^T\overline{\mathbf{e}}_{k-1}|^2}{\|\mathbf{A}\bar{\mathbf{s}}_k\|_2^2-\|\bar{\mathbf{d}}_{k-1}\|_2^2}\\
&=\|\overline{\mathbf{e}}_{k-1}\|_2^2-\frac{|\bar{\mathbf{s}}_k^T\mathbf{A}^T\overline{\mathbf{e}}_{k-1}|^2}{\sin^2(\tilde{\theta}_{k-1})\|\mathbf{A}\bar{\mathbf{s}}_k\|_2^2},
\end{aligned}
$$
where $\tilde{\theta}_{k-1}$ is the angle between the vector $\mathbf{A}\bar{\mathbf{s}}_k$ and the column space of $\bar{\mathbf{P}}_{k-1}$. The result follows from taking the conditional expectation of block $\tilde{{\tau}}_k$ over the entire history $\mathcal{F}_{k-1} = \{\chi_1,\cdots, \chi_{k-1}\}$, which yields
$$
\begin{aligned}
\mathbb{E}_{\tilde{{\tau}}_k}[\|\overline{\mathbf{e}}_k\|_2^2|\mathcal{F}_{k-1}]&\leq\|\overline{\mathbf{e}}_{k-1}\|_2^2-\mathbb{E}_{\tilde{{\tau}}_k}\left[\frac{|\bar{\mathbf{s}}_k^T\mathbf{A}^T\overline{\mathbf{e}}_{k-1}|^2}{\sin^2(\tilde{\theta}_{k-1})\|\mathbf{A}\bar{\mathbf{s}}_k\|_2^2}\Bigg|\mathcal{F}_{k-1}\right]\\
&\leq\|\overline{\mathbf{e}}_{k-1}\|_2^2-\frac{\mathbb{E}_{\tilde{{\tau}}_k}[|\bar{\mathbf{s}}_k^T\mathbf{A}^T\overline{\mathbf{e}}_{k-1}|^2|\mathcal{F}_{k-1}]}{\sin^2(\tilde{\theta}_{k-1})\mathbb{E}_{\tilde{{\tau}}_k}[\|\mathbf{A}\bar{\mathbf{s}}_k\|_2^2|\mathcal{F}_{k-1}]}\\
&\leq\|\overline{\mathbf{e}}_{k-1}\|_2^2-\frac{l^2\|\mathbf{A}^T\overline{\mathbf{e}}_{k-1}\|_2^4}{\sin^2(\tilde{\theta}_{k-1})l\lambda_{max}^2(\mathbf{A}\mathbf{A}^T)\|\overline{\mathbf{e}}_{k-1}\|_2^2}\\
&\leq(1-\frac{l\lambda_{min}^2(\mathbf{A}\mathbf{A}^T)}{\sin^2(\tilde{\theta}_{k-1})\lambda_{max}^2(\mathbf{A}\mathbf{A}^T)})\|\overline{\mathbf{e}}_{k-1}\|_2^2.
\end{aligned}
$$
This completes the proof.
\end{proof}

From Theorem \ref{The-4}, we give the following theorem for guaranteeing the convergence of Algorithm \ref{algo-2-2}.
\begin{theorem}\label{The-5}
{\rm Let $\mathbf{x}_{\star}=\mathbf{A}^{\dagger}\mathbf{b}$ be the least-norm least-squares solution of the inconsistent linear system \eqref{1-1}. Then the iterative sequence $\{\mathbf{x}_{k}\}$ generated by Algorithm \ref{algo-2-2} converges to $\mathbf{x}_{\star }$ for any initial vector $\mathbf{x}_0$ in the column space of $\mathbf{A}^T$. Moreover, the corresponding error norm in expectation yields
}
\begin{equation}\label{3-21}
\mathbb{E}\|\mathbf{x}_{k}-\mathbf{x}_{\star}\|_2^2\leq\left(\hat{\alpha}^k+\hat{\beta}^{k-1}\frac{\lambda_{max}(\mathbf{A}^T\mathbf{A})}{\lambda_{min}(\mathbf{A}_{\tau_k}^T\mathbf{A}_{\tau_k})}\right)\|\mathbf{x}_{\star}\|_2^2,    
\end{equation}
\rm where $\hat{\alpha}=\max\limits_k\frac{\lambda_{min}(\mathbf{A}_{\tau_k}^T\mathbf{A}_{\tau_k})}{\lambda_{max}(\mathbf{A}_{\tau_k}\mathbf{A}_{\tau_k}^T)}$.
\end{theorem}
\begin{proof}
From Algorithm \ref{algo-2-2}, we have
\begin{equation}\label{2-7}
\begin{aligned}
\mathbf{x}_{k}-\mathbf{x}_{\star}&=\mathbf{x}_{k-1}-\mathbf{x}_{\star}+\frac{\mathbf{s}_k^T(\mathbf{b-\mathbf{z}_{k-1}-\mathbf{A}\mathbf{x}_{k-1})}}{\|\mathbf{A}^T\mathbf{s}_k\|_2^2\|_2^2}\mathbf{A}^T\mathbf{s}_k\\
&=\left(\mathbf{I}-\frac{\mathbf{A}^T\mathbf{s}_k\mathbf{s}_k^T\mathbf{A}}{\|\mathbf{A}^T\mathbf{s}_k\|_2^2}\right)(\mathbf{x}_{k-1}-\mathbf{x}_{\star})-\frac{\mathbf{s}_k^T(\mathbf{b}_{\mathcal{R}(\mathbf{A})^{\bot}}-\mathbf{z}_{k-1})\mathbf{A}^T\mathbf{s}_k}{\|\mathbf{A}^T\mathbf{s}_k\|_2^2}.
\end{aligned}
\end{equation}
It follows from Eq.\eqref{2-7} that
$$
\begin{aligned}
&\|\mathbf{x}_{k}-\mathbf{x}_{\star}\|_2^2=\left\|\left(\mathbf{I}-\frac{\mathbf{A}^T\mathbf{s}_k\mathbf{s}_k^T\mathbf{A}}{\|\mathbf{A}^T\mathbf{s}_k\|_2^2}\right)(\mathbf{x}_{k-1}-\mathbf{x}_{\star})\right\|_2^2+\left\|\frac{\mathbf{s}_k^T(\mathbf{b}_{\mathcal{R}(\mathbf{A})^{\bot}}-\mathbf{z}_{k-1})\mathbf{A}^T\mathbf{s}_k}{\|\mathbf{A}^T\mathbf{s}_k\|_2^2}\right\|_2^2\\
&=\|\mathbf{x}_{k-1}-\mathbf{x}_{\star}\|_2^2-\frac{|\mathbf{s}_k^T(\mathbf{b}_{\mathcal{R}(\mathbf{A})}-\mathbf{A}\mathbf{x}_{k-1})|^2-|\mathbf{s}_k^T(\mathbf{b}_{\mathcal{R}(\mathbf{A})^{\bot}}-\mathbf{z}_{k-1})|^2}{\|\mathbf{A}^T\mathbf{s}_k\|_2^2}\\
&=\|\mathbf{x}_{k-1}-\mathbf{x}_{\star}\|_2^2-\frac{\mathbf{s}_k^T(\mathbf{b}-\mathbf{A}\mathbf{x}_{k-1}-\mathbf{z}_{k-1})\mathbf{s}_k^T(\mathbf{b}_{\mathcal{R}(\mathbf{A})}-\mathbf{A}\mathbf{x}_{k-1}-(\mathbf{b}_{\mathcal{R}(\mathbf{A})^{\bot}}-\mathbf{z}_{k-1}))}{\|\mathbf{A}^T\mathbf{s}_k\|_2^2}\\
&\leq\|\mathbf{x}_{k-1}-\mathbf{x}_{\star}\|_2^2-\frac{\mathbf{s}_k^T(\mathbf{b}_{\mathcal{R}(\mathbf{A})}-\mathbf{A}\mathbf{x}_{k-1}-(\mathbf{b}_{\mathcal{R}(\mathbf{A})^{\bot}}-\mathbf{z}_{k-1}))\|\mathbf{s}_k\|_2^2}{\lambda_{max}(\mathbf{A}_{\tau_k}\mathbf{A}_{\tau_k}^T)\|\mathbf{s}_k\|_2^2}\\
&=\|\mathbf{x}_{k-1}-\mathbf{x}_{\star}\|_2^2-\sum_{i\in \tau_k}\frac{(\mathbf{b}_{\mathcal{R}(\mathbf{A})}^{(i)}-\mathbf{A}^{(i)}\mathbf{x}_{k-1}+\mathbf{b}_{\mathcal{R}(\mathbf{A})^{\bot}}^{(i)}-\mathbf{z}_{k-1}^{(i)})}{\lambda_{max}(\mathbf{A}_{\tau_k}\mathbf{A}_{\tau_k}^T)}\\
&\quad(\mathbf{b}_{\mathcal{R}(\mathbf{A})}^{(i)}-\mathbf{A}^{(i)}\mathbf{x}_{k-1}-(\mathbf{b}_{\mathcal{R}(\mathbf{A})^{\bot}}^{(i)}-\mathbf{z}_{k-1}^{(i)}))\\
&=\|\mathbf{x}_{k-1}-\mathbf{x}_{\star}\|_2^2-\frac{\|\mathbf{A}_{\tau_k}(\mathbf{x}_{k-1}-\mathbf{x}_{\star})\|_2^2-\|\mathbf{E}_{\tau_k}(\mathbf{b}_{\mathcal{R}(\mathbf{A})^{\bot}}-\mathbf{z}_{k-1})\|_2^2}{\lambda_{max}(\mathbf{A}_{\tau_k}\mathbf{A}_{\tau_k}^T)}\\
&\leq\|\mathbf{x}_{k-1}-\mathbf{x}_{\star}\|_2^2-\frac{\lambda_{min}(\mathbf{A}_{\tau_k}^T\mathbf{A}_{\tau_k})\|\mathbf{x}_{k-1}-\mathbf{x}_{\star}\|_2^2-\|\mathbf{b}_{\mathcal{R}(\mathbf{A})^{\bot}}-\mathbf{z}_{k-1}\|_2^2}{\lambda_{max}(\mathbf{A}_{\tau_k}\mathbf{A}_{\tau_k}^T)}\\
&=\left(1-\frac{\lambda_{min}(\mathbf{A}_{\tau_k}^T\mathbf{A}_{\tau_k})}{\lambda_{max}(\mathbf{A}_{\tau_k}\mathbf{A}_{\tau_k}^T}\right)\|\mathbf{x}_{k-1}-\mathbf{x}_{\star}\|_2^2+\frac{\|\mathbf{b}_{\mathcal{R}(\mathbf{A})^{\bot}}-\mathbf{z}_{k-1}\|_2^2}{\lambda_{max}(\mathbf{A}_{\tau_k}\mathbf{A}_{\tau_k}^T)}.
\end{aligned}
$$
Define $\hat{\alpha}=\max\limits_k(1-\frac{\lambda_{min}(\mathbf{A}_{\tau_k}^T\mathbf{A}_{\tau_k})}{\lambda_{max}(\mathbf{A}_{\tau_k}\mathbf{A}_{\tau_k}^T})$. According to Theorem \ref{The-4}, the above relation further gives
$$
\begin{aligned}
\mathbb{E}\|\mathbf{x}_{k}-\mathbf{x}_{\star}\|_2^2&\leq \hat{\alpha}\|\mathbf{x}_{k-1}-\mathbf{x}_{\star}\|_2^2+\mathbb{E} \frac{\|\mathbf{b}_{\mathcal{R}(\mathbf{A})^{\bot}}-\mathbf{z}_{k-1}\|_2^2}{\lambda_{max}(\mathbf{A}_{\tau_k}\mathbf{A}_{\tau_k}^T)}\\
&\leq\hat{\alpha}\|\mathbf{x}_{k-1}-\mathbf{x}_{\star}\|_2^2+\frac{\hat{\beta}^{k-1}\|\mathbf{b}_{\mathcal{R}(\mathbf{A})}\|_2^2}{\lambda_{max}(\mathbf{A}_{\tau_k}\mathbf{A}_{\tau_k}^T)}\\
&\leq\hat{\alpha}^2\|\mathbf{x}_{k-2}-\mathbf{x}_{\star}\|_2^2+\hat{\alpha}\frac{\hat{\beta}^{k-1}\|\mathbf{b}_{\mathcal{R}(\mathbf{A})}\|_2^2}{\lambda_{max}(\mathbf{A}_{\tau_k}\mathbf{A}_{\tau_k}^T)}+\frac{\hat{\beta}^{k-1}\|\mathbf{b}_{\mathcal{R}(\mathbf{A})}\|_2^2}{\lambda_{max}(\mathbf{A}_{\tau_k}\mathbf{A}_{\tau_k}^T)}\\
&\leq\hat{\alpha}^k\|\mathbf{x}_{\star}\|_2^2+\sum_{l=0}^{k-1}\hat{\alpha}^{l}\frac{\hat{\beta}^{k-1}\|\mathbf{b}_{\mathcal{R}(\mathbf{A})}\|_2^2}{\lambda_{max}(\mathbf{A}_{\tau_k}\mathbf{A}_{\tau_k}^T)}\\
&\leq\hat{\alpha}^k\|\mathbf{x}_{\star}\|_2^2+\frac{\hat{\beta}^{k-1}\|\mathbf{b}_{\mathcal{R}(\mathbf{A})}\|_2^2}{(1-\hat{\alpha})\lambda_{max}(\mathbf{A}_{\tau_k}\mathbf{A}_{\tau_k}^T)}\\
\end{aligned}
$$
$$
\begin{aligned}
&\leq\hat{\alpha}^k\|\mathbf{x}_{\star}\|_2^2+\frac{\hat{\beta}^{k-1}\lambda_{max}(\mathbf{A}^T\mathbf{A})}{\lambda_{min}(\mathbf{A}_{\tau_k}^T\mathbf{A}_{\tau_k})}\|\mathbf{x}_{\star}\|_2^2\\
&=\left(\hat{\alpha}^k+\hat{\beta}^{k-1}\frac{\lambda_{max}(\mathbf{A}^T\mathbf{A})}{\lambda_{min}(\mathbf{A}_{\tau_k}^T\mathbf{A}_{\tau_k})}\right)\|\mathbf{x}_{\star}\|_2^2.
\end{aligned}
$$
This completes the proof.
\end{proof}

\section{RPLSS with coordinate descent method}\label{section-3}
Note that the column-action solver is an attractive alternative for linear systems \eqref{1-1}, among which the randomized coordinate descent (RCD) method, also called randomized Gauss-Seidel, is well-established. Unlike the RK, which sequentially projects the iterate onto the hyperplanes associated with individual rows, the RCD operates on the column space by successively minimizing the objective residual norm. This structural difference enables RCD to converge unconditionally to a least-squares solution for inconsistent systems, where the RK method otherwise suffers from non-convergent oscillations. However, the RCD updates only one component per iteration and discards historical information, which may hinder its convergence on ill-conditioned and large-scale problems. To develop a column-action algorithm within the randomized PLSS framework, we redefine the selection of the sketch column $\mathbf{s}_k$ and the weight matrix $\mathbf{W}$. This gives a class of alternative randomized PLSS algorithms based on the column actions.
\subsection{RPLSS-CD}\label{subsection-2-3}
In the RPLSS framework, let $\mathbf{s}_k=\mathbf{A}\mathbf{e}_i$ with $\mathbf{e}_i$
being sampled with probability $\mathbb{P}$(column=$i)=$$\frac{\|\mathbf{A}_{(i)}\|_2^2}{\|\mathbf{A}\|_F^2}$ and $\mathbf{W}=(\mathbf{A}^T\mathbf{A})^{-1}$. Then
we derive the following RPLSS coordinate descent (RPLSS-CD) method for solving Eq.\eqref{1-1}.
\begin{algorithm}
	\caption{RPLSS-CD}
	\label{algo-3-1}
	\begin{algorithmic}[1]
        \REQUIRE $\mathbf{A}, \mathbf{b},\mathbf{W}=(\mathbf{A}^T\mathbf{A})^{-1}$\\
        \ENSURE $\mathbf{x}_{k}$\\
        \STATE
       For $k=0, 1, 2,\cdots $, do:\\
         \STATE
 Select $i\in[n]$ with probability $\mathbb{P}$(column=$i)=$$\frac{\|\mathbf{A}_{(i)}\|_2^2}{\|\mathbf{A}\|_F^2}$
 \STATE
   Compute 
   $$\mathbf{s}_k=\mathbf{A}\mathbf{e}_i$$
    \STATE
   Compute 
   $$\mathbf{p}_k=\gamma_{k-1}(\mathbf{W}\mathbf{A}^T\mathbf{s}_k-\mathbf{P}_{k-1}\mathbf{\Theta}_{k-1} ^{-1}\mathbf{P}_{k-1}^T\mathbf{A}^T\mathbf{s}_k)$$
   where $\gamma_{k-1}=\frac{\mathbf{s}_k^T(\mathbf{b}-\mathbf{A}\mathbf{x}_{k-1})}{\mathbf{s}_k^T\mathbf{A}\mathbf{W}\mathbf{A}^T\mathbf{s}_k-\|\mathbf{d}_{k-1}\|_2^2}$ with $\mathbf{d}_{k-1}=\mathbf{\Theta}_{k-1} ^{-1/2}\mathbf{P}_{k-1}^T\mathbf{A}^T\mathbf{s}_k$ and $\mathbf{\Theta}_{k-1}=\mathbf{P}_{k-1}^T\mathbf{W}^{-1}\mathbf{P}_{k-1}$.
            \STATE
            Set $\mathbf{x}_{k} =\mathbf{x}_{k-1}+\mathbf{p}_k$.

	\end{algorithmic}  
\end{algorithm}
\begin{rem}\label{Rem-2}
\rm In fact, for step $2$ of the Algorithm \ref{algo-3-1}, different random column selection rules will result in different randomized PLSS coordinate descent methods. If the $i$-th column is selected by the greedy strategy, then an RPLSS with the greedy coordinate descent method is presented.
\end{rem}

For the convergence property of the RPLSS-CD method, we can establish the following theorem.
\begin{theorem}\label{The-3-1}
{\rm Let $\mathbf{x}_{\star}=\mathbf{A}^{\dagger}\mathbf{b}$ be the unique least-squares solution of the linear system \eqref{1-1}, where $\mathbf{A}\in \mathbb{R}^{m\times n}$ is a rectangular matrix of full column rank and $m\geq n$. Then the iterative sequence $\{\mathbf{x}_{k}\}$ generated by Algorithm \ref{algo-3-1} converges to $\mathbf{x}_{\star }$ for any initial vector $\mathbf{x}_0$ in expectation. Moreover, the corresponding error norm in expectation yields
}
\begin{equation}\label{2-8}
\mathbb{E}_k\|\mathbf{x}_{k}-\mathbf{x}_{\star}\|_{\mathbf{A}^T\mathbf{A}}^2\leq\left(1-\frac{\lambda_{min}(\mathbf{A}\mathbf{A}^T)}{\sin^2(\varphi_{k-1})\|\mathbf{A}\|_F^2}\right)\|\mathbf{x}_{k-1}-\mathbf{x}_{\star}\|_{\mathbf{A}^T\mathbf{A}}^2.   
\end{equation}
\end{theorem}
\begin{proof}
As was seen from Algorithm \ref{algo-3-1}, it follows that
$$
\begin{aligned}
\mathbf{A}\overline{\mathbf{x}}_k&=\mathbf{A}\left(\overline{\mathbf{x}}_{k-1}+\frac{\mathbf{s}_k^T(\mathbf{b}-\mathbf{A}\mathbf{x}_{k-1})\mathbf{e}_i-\mathbf{s}_k^T(\mathbf{b}-\mathbf{A}\mathbf{x}_{k-1})\mathbf{P}_{k-1}\mathbf{\Theta}_{k-1} ^{-1}\mathbf{P}_{k-1}^T\mathbf{A}^T\mathbf{A}\mathbf{e}_i}{\|\mathbf{A}_{(i)}\|_2^2-\|\mathbf{d}_{k-1}\|_2^2}\right)\\
&=(\mathbf{I}-P)\mathbf{A}\overline{\mathbf{x}}_{k-1}
\end{aligned}
$$
in which $\overline{\mathbf{x}}_k=\mathbf{x}_k-\mathbf{x}_{\star}$ and $P=\frac{\mathbf{A}_{(i)}\mathbf{A}_{(i)}^T-\mathbf{A}\mathbf{P}_{k-1}\mathbf{\Theta}_{k-1} ^{-1}\mathbf{P}_{k-1}^T\mathbf{A}^T\mathbf{A}_{(i)}\mathbf{A}_{(i)}^T}{\|\mathbf{A}_{(i)}\|_2^2-\|\mathbf{d}_{k-1}\|_2^2}$. It is easy to verify that $(\mathbf{I}-P)^2=\mathbf{I}-P$. The above relation gives
$$
\begin{aligned}
\mathbb{E}_k\|\mathbf{A}\overline{\mathbf{x}}_k\|_2^2&=\mathbb{E}_k\|(\mathbf{I}-P)\mathbf{A}\overline{\mathbf{x}}_{k-1}\|_2^2\\
&=\|\mathbf{A}\overline{\mathbf{x}}_{k-1}\|_2^2-\mathbb{E}_k\|P\mathbf{A}\overline{\mathbf{x}}_{k-1}\|_2^2\\
&=\|\mathbf{A}\overline{\mathbf{x}}_{k-1}\|_2^2-\mathbb{E}_k\left\|\frac{(\mathbf{I}-\mathbf{A}\mathbf{P}_{k-1}\mathbf{\Theta}_{k-1} ^{-1}\mathbf{P}_{k-1}^T\mathbf{A}^T)\mathbf{A}_{(i)}\mathbf{A}_{(i)}^T}{\|\mathbf{A}_{(i)}\|_2^2-\|\mathbf{d}_{k-1}\|_2^2}\mathbf{A}\overline{\mathbf{x}}_{k-1}\right\|_2^2\\
&=\|\mathbf{A}\overline{\mathbf{x}}_{k-1}\|_2^2-\mathbb{E}_k\frac{\|(\mathbf{I}-\mathbf{A}\mathbf{P}_{k-1}\mathbf{\Theta}_{k-1} ^{-1}\mathbf{P}_{k-1}^T\mathbf{A}^T)\mathbf{A}_{(i)}\mathbf{A}_{(i)}^T\mathbf{A}\overline{\mathbf{x}}_{k-1}\|_2^2}{(\|\mathbf{A}_{(i)}\|_2^2-\|\mathbf{d}_{k-1}\|_2^2)^2}\\
&=\|\mathbf{A}\overline{\mathbf{x}}_{k-1}\|_2^2-\mathbb{E}_k\frac{(\|\mathbf{A}_{(i)}\|_2^2-\|\mathbf{d}_{k-1}\|_2^2)|\mathbf{A}_{(i)}^T\mathbf{A}\overline{\mathbf{x}}_{k-1}|^2}{(\|\mathbf{A}_{(i)}\|_2^2-\|\mathbf{d}_{k-1}\|_2^2)^2}\\
&=\|\mathbf{A}\overline{\mathbf{x}}_{k-1}\|_2^2-\sum_{i=1}^{n}\frac{\|\mathbf{A}_{(i)}\|_2^2}{\|\mathbf{A}\|_F^2}\frac{|\mathbf{A}_{(i)}^T\mathbf{A}\overline{\mathbf{x}}_{k-1}|^2}{\|\mathbf{A}_{(i)}\|_2^2-\|\mathbf{d}_{k-1}\|_2^2}\\
&=\|\mathbf{A}\overline{\mathbf{x}}_{k-1}\|_2^2-\sum_{i=1}^{n}\frac{\|\mathbf{A}_{(i)}\|_2^2}{\|\mathbf{A}\|_F^2}\frac{|\mathbf{A}_{(i)}^T\mathbf{A}\overline{\mathbf{x}}_{k-1}|^2}{\sin^2(\varphi_{k-1})\|\mathbf{A}_{(i)}\|_2^2}\\
&=\|\mathbf{A}\overline{\mathbf{x}}_{k-1}\|_2^2-\frac{\|\mathbf{A}^T\mathbf{A}\overline{\mathbf{x}}_{k-1}\|_2^2}{\sin^2(\varphi_{k-1})\|\mathbf{A}\|_F^2}\\
&\leq(1-\frac{\lambda_{min}(\mathbf{A}\mathbf{A}^T)}{\sin^2(\varphi_{k-1})\|\mathbf{A}\|_F^2})\|\mathbf{A}\overline{\mathbf{x}}_{k-1}\|_2^2,
\end{aligned}
$$
where $\varphi_{k-1}$ is the angle between the vectors $\mathbf{A}_{(i)}$ and $\mathbf{A}\mathbf{P}_{k-1}$.  
\end{proof}

\begin{rem}\label{Rem-6}
{\rm As $0<\sin^2(\varphi_{k-1})\leq1$, the upper bounds for the convergence rates of RPLSS-CD given in Theorem \ref{The-5} satisfy
$$1-\frac{\lambda_{min}(\mathbf{A}\mathbf{A}^T)}{\sin^2(\varphi_{k-1})\|\mathbf{A}\|_F^2}\leq 1-\frac{\lambda_{min}(\mathbf{A}\mathbf{A}^T)}{\|\mathbf{A}\| _F^2},
$$
which converges faster than the RCD method.
}    
\end{rem}

Next, we give a more generalized RPLSS-CD method, which is more stable than RPLSS-CD.
\begin{algorithm}
	\caption{RPLSS-CD($|\tau_k|$)}
	\label{algo-3-2}
	\begin{algorithmic}[1]
        \REQUIRE $\mathbf{A}, \mathbf{b},\mathbf{W}=(\mathbf{A}^T\mathbf{A})^{-1}$\\
        \ENSURE $\mathbf{x}_{k}$\\
        \STATE
       For $k=0, 1, 2,\cdots,n $, do:\\
         \STATE
 Let $\tilde{\tau}_k=(i_1,i_2,\dots,i_t)$ and $i_t\in[n]$ is drawn independently with probability $\mathbb{P}$(column=$i_t)=$$\frac{\|A_{(i_t)}\|_2^2}{\|A\|_F^2}$. Then define $\tau_k$ to be the collection of unique elements in $\tilde{\tau}_k$.
 \STATE
   Compute 
   $$\mathbf{s}_k=\sum\limits_{i\in\tau_k}\mathbf{A}\mathbf{e}_i$$
    \STATE
   Compute 
   $$\mathbf{p}_k=\gamma_{k-1}(\mathbf{W}\mathbf{A}^T\mathbf{s}_k-\mathbf{P}_{k-1}\mathbf{\Theta}_{k-1} ^{-1}\mathbf{P}_{k-1}^T\mathbf{A}^T\mathbf{s}_k)$$
   where $\gamma_{k-1}=\frac{\mathbf{s}_k^T(\mathbf{b}-\mathbf{A}\mathbf{x}_{k-1})}{\mathbf{s}_k^T\mathbf{A}\mathbf{W}\mathbf{A}^T\mathbf{s}_k-\|\mathbf{d}_{k-1}\|_2^2}$ with $\mathbf{d}_{k-1}=\mathbf{\Theta}_{k-1} ^{-1/2}\mathbf{P}_{k-1}^T\mathbf{A}^T\mathbf{s}_k$ and $\mathbf{\Theta}_{k-1}=\mathbf{P}_{k-1}^T\mathbf{W}^{-1}\mathbf{P}_{k-1}$.
            \STATE
            Set $\mathbf{x}_{k} =\mathbf{x}_{k-1}+\mathbf{p}_k$.

	\end{algorithmic}  
\end{algorithm}

Regarding the convergence of the Algorithm \ref{algo-3-2}, we establish the following theorem.

\begin{theorem}\label{The-3-2}
{\rm  Let $\mathbf{x}_{\star}=\mathbf{A}^{\dagger}\mathbf{b}$ be the unique least-squares solution of the linear system \eqref{1-1}, where $\mathbf{A}\in \mathbb{R}^{m\times n}$ is a rectangular matrix of full column rank and $m\geq n$. Then the iterative sequence $\{\mathbf{x}_{k}\}$ generated by Algorithm \ref{algo-3-2} converges to $\mathbf{x}_{\star }$ for any initial vector $\mathbf{x}_0\in \mathbb{R}^n$. Moreover, the corresponding error norm yields
}
\begin{equation}\label{3-10}
\|\mathbf{x}_{k}-\mathbf{x}_{\star}\|_{\mathbf{A}^T\mathbf{A}}^2\leq\left(1-\frac{\lambda_{min}(\mathbf{A}^T\mathbf{A})}{|\tau_k|\sin^2(\bar{\varphi}_{k-1})\|\mathbf{A}_{\tau_k}\|_F^2}\right)\|\mathbf{x}_{k-1}-\mathbf{x}_{\star}\|_{\mathbf{A}^T\mathbf{A}}^2,    
\end{equation}
\end{theorem}
\begin{proof}
 From Algorithm \ref{algo-3-2}, we observe that
$$\begin{small}
\begin{aligned}
\mathbf{A}\overline{\mathbf{x}}_k&=\mathbf{A}\left(\overline{\mathbf{x}}_{k-1}+\frac{\mathbf{s}_k^T(\mathbf{b}-\mathbf{A}\mathbf{x}_{k-1})\sum\limits_{i\in\tau_k}\mathbf{e}_i-\mathbf{s}_k^T(\mathbf{b}-\mathbf{A}\mathbf{x}_{k-1})\mathbf{P}_{k-1}\mathbf{\Theta}_{k-1} ^{-1}\mathbf{P}_{k-1}^T\mathbf{A}^T\mathbf{A}\sum\limits_{i\in\tau_k}\mathbf{e}_i}{\|\mathbf{s}_{k}\|_2^2-\|\mathbf{d}_{k-1}\|_2^2}\right)\\
&=(\mathbf{I}-P)\mathbf{A}\overline{\mathbf{x}}_{k-1},
\end{aligned}\end{small}
$$
in which $\overline{\mathbf{x}}_k=\mathbf{x}_k-\mathbf{x}_{\star}$ and $P=\frac{\mathbf{s}_{k}\mathbf{s}_{k}^T-\mathbf{A}\mathbf{P}_{k-1}\mathbf{\Theta}_{k-1} ^{-1}\mathbf{P}_{k-1}^T\mathbf{A}^T\mathbf{s}_{k}\mathbf{s}_{k}^T}{\|\mathbf{s}_{k}\|_2^2-\|\mathbf{d}_{k-1}\|_2^2}$. It is easy to verify that $(\mathbf{I}-P)^2=\mathbf{I}-P$. Therefore,
$$
\begin{aligned}
\|\mathbf{A}\overline{\mathbf{x}}_k\|_2^2&=\|(\mathbf{I}-P)\mathbf{A}\overline{\mathbf{x}}_{k-1}\|_2^2\\
&=\|\mathbf{A}\overline{\mathbf{x}}_{k-1}\|_2^2-\|P\mathbf{A}\overline{\mathbf{x}}_{k-1}\|_2^2\\
&=\|\mathbf{A}\overline{\mathbf{x}}_{k-1}\|_2^2-\left\|\frac{(\mathbf{I}-\mathbf{A}\mathbf{P}_{k-1}\mathbf{\Theta}_{k-1} ^{-1}\mathbf{P}_{k-1}^T\mathbf{A}^T)\mathbf{s}_{k}\mathbf{s}_{k}^T}{\|\mathbf{s}_{k}\|_2^2-\|\mathbf{d}_{k-1}\|_2^2}\mathbf{A}\overline{\mathbf{x}}_{k-1}\right\|_2^2\\
&=\|\mathbf{A}\overline{\mathbf{x}}_{k-1}\|_2^2-\frac{\|(\mathbf{I}-\mathbf{A}\mathbf{P}_{k-1}\mathbf{\Theta}_{k-1} ^{-1}\mathbf{P}_{k-1}^T\mathbf{A}^T)\mathbf{s}_{k}\mathbf{s}_{k}^T\mathbf{A}\overline{\mathbf{x}}_{k-1}\|_2^2}{(\|\mathbf{s}_{k}\|_2^2-\|\mathbf{d}_{k-1}\|_2^2)^2}\\
&=\|\mathbf{A}\overline{\mathbf{x}}_{k-1}\|_2^2-\frac{(\|\mathbf{s}_{k}\|_2^2-\|\mathbf{d}_{k-1}\|_2^2)|\mathbf{s}_{k}^T\mathbf{A}\overline{\mathbf{x}}_{k-1}|^2}{(\|\mathbf{s}_{k}\|_2^2-\|\mathbf{d}_{k-1}\|_2^2)^2}\\
\end{aligned}
$$
$$
\begin{aligned}
&=\|\mathbf{A}\overline{\mathbf{x}}_{k-1}\|_2^2-\frac{|\mathbf{s}_{k}^T\mathbf{A}\overline{\mathbf{x}}_{k-1}|^2}{\|\mathbf{s}_{k}\|_2^2-\|\mathbf{d}_{k-1}\|_2^2}\\
&=\|\mathbf{A}\overline{\mathbf{x}}_{k-1}\|_2^2-\frac{|\mathbf{s}_{k}^T\mathbf{A}\overline{\mathbf{x}}_{k-1}|^2}{\sin^2(\bar{\varphi}_{k-1})\|\mathbf{s}_{k}\|_2^2}\\
&\leq\|\mathbf{A}\overline{\mathbf{x}}_{k-1}\|_2^2-\frac{\|\mathbf{A}_{\tau_k}^T\mathbf{A}\overline{\mathbf{x}}_{k-1}\|_2^2}{\sin^2(\bar{\varphi}_{k-1})|\tau_k|\|\mathbf{A}_{\tau_k}\|_F^2}\\
&\leq(1-\frac{\lambda_{min}(\mathbf{A}^T\mathbf{A})}{|\tau_k|\sin^2(\bar{\varphi}_{k-1})\|\mathbf{A}_{\tau_k}\|_F^2})\|\mathbf{A}\overline{\mathbf{x}}_{k-1}\|_2^2,
\end{aligned}
$$
where $\bar{\varphi}_{k-1}$ is the angle between the vectors $\mathbf{s}_{k}$ and $\mathbf{A}\mathbf{P}_{k-1}$.  
\end{proof}

\subsection{RPLSS-GCD}\label{subsection-2-4}
As presented in the previous section, the RPLSS‑CD randomly selects a single column per iteration to construct the sketch. While this strategy is memory‑efficient, it updates only one coordinate at a time and ignores the global error information contained in the residual vector. To making fully use of the residuals, we construct a sketch column $\mathbf{s}_k$ by randomly linearly combining the residuals with the standard basis vectors, which can vastly accelerate the convergence of ill-conditioned problems. Specifically, in the RPLSS framework, we set $\mathbf{s}_k=\sum_{i\in \tau_k}\mathbf{A}\mathbf{r}^{(i)}_{k-1}\mathbf{e}_i$ and $\mathbf{W}=(\mathbf{A}^T\mathbf{A})^{-1}$, then develop the following RPLSS Gaussian coordinate descent (RPLSS-GCD) method for solving Eq.\eqref{1-1}.
\begin{algorithm}
	\caption{RPLSS-GCD}
	\label{algo-2-5}
	\begin{algorithmic}[1]
        \REQUIRE $\mathbf{A}, \mathbf{b},\mathbf{W}=(\mathbf{A}^T\mathbf{A})^{-1}$\\
        \ENSURE $\mathbf{x}_{k}$\\
        \STATE
       For $k=0, 1, 2,\cdots $, do:\\
         \STATE
Generate an indicator set $\tau_{k}$, i.e., choosing $l n$ columns of $\mathbf{A}$ by using the simple random sampling, where $0<l<1$
 \STATE
   Compute 
   $$\mathbf{s}_k=\sum_{i\in \tau_{k}}\mathbf{A}\mathbf{r}^{(i)}_{k-1}\mathbf{e}_i$$
   where $\mathbf{r}_{k-1}=\mathbf{A}^T(\mathbf{b}-\mathbf{A}\mathbf{x}_{k-1})$
    \STATE
   Compute 
   $$\mathbf{p}_k=\gamma_{k-1}(\mathbf{W}\mathbf{A}^T\mathbf{s}_k-\mathbf{P}_{k-1}\mathbf{\Theta}_{k-1} ^{-1}\mathbf{P}_{k-1}^T\mathbf{A}^T\mathbf{s}_k)$$
   where $\gamma_{k-1}=\frac{\mathbf{s}_k^T(\mathbf{b}-\mathbf{A}\mathbf{x}_{k-1})}{\mathbf{s}_k^T\mathbf{A}\mathbf{W}\mathbf{A}^T\mathbf{s}_k-\|\mathbf{d}_{k-1}\|_2^2}$ with $\mathbf{d}_{k-1}=\mathbf{\Theta}_{k-1} ^{-1/2}\mathbf{P}_{k-1}^T\mathbf{A}^T\mathbf{s}_k$ and $\mathbf{\Theta}_{k-1}=\mathbf{P}_{k-1}^T\mathbf{W}^{-1}\mathbf{P}_{k-1}$.
            \STATE
            Set $\mathbf{x}_{k} =\mathbf{x}_{k-1}+\mathbf{p}_k$.

	\end{algorithmic}  
\end{algorithm}

Next, we show the exponential convergence of the Algorithm \ref{algo-2-5}.
\begin{theorem}\label{The-3-3}
{\rm Let $\mathbf{x}_{\star}=\mathbf{A}^{\dagger}\mathbf{b}$ be the unique least-squares solution of the inconsistent linear system \eqref{1-1}, where $\mathbf{A}\in \mathbb{R}^{m\times n}$ is a rectangular matrix of full column rank and $m\geq n$. Then the iterative sequence $\{\mathbf{x}_{k}\}$ generated by Algorithm \ref{algo-2-5} converges to $\mathbf{x}_{\star }$ for any initial vector $\mathbf{x}_0$. Moreover, the corresponding error norm in expectation yields
}
\begin{equation}\label{2-10}
\mathbb{E}_k\|\mathbf{x}_{k}-\mathbf{x}_{\star}\|_{\mathbf{A}^T\mathbf{A}}^2\leq\left(1-\frac{l\lambda_{min}^2(\mathbf{A}^T\mathbf{A})}{\sin^2(\vartheta)\lambda_{max}^2(\mathbf{A}^T\mathbf{A})}\right)\|\mathbf{x}_{k-1}-\mathbf{x}_{\star}\|_{\mathbf{A}^T\mathbf{A}}^2.  
\end{equation}
\end{theorem}
\begin{proof}
Observe from Algorithm \ref{algo-2-5} that the error $\overline{\mathbf{x}}_k=\mathbf{x}_{k-1}-\mathbf{x}_{\star}$ yields
$$
\begin{aligned}
\mathbf{A}\overline{\mathbf{x}}_k&=\mathbf{A}\left(\overline{\mathbf{x}}_{k-1}-\frac{\mathbf{s}_k^T\mathbf{A}\overline{\mathbf{x}}_{k-1}\mathbf{W}\mathbf{A}^T\mathbf{s}_k-\mathbf{s}_k^T\mathbf{A}\overline{\mathbf{x}}_{k-1}\mathbf{P}_{k-1}\mathbf{\Theta}_{k-1} ^{-1}\mathbf{P}_{k-1}^T\mathbf{A}^T\mathbf{s}_k}{\mathbf{s}_k^T\mathbf{A}\mathbf{W}\mathbf{A}^T\mathbf{s}_k-\|\mathbf{d}_{k-1}\|_2^2}\right)\\
&=(\mathbf{I}-P)\mathbf{A}\overline{\mathbf{x}}_{k-1}
\end{aligned}
$$
where $P=\frac{\mathbf{s}_k\mathbf{s}_k^T-\mathbf{A}\mathbf{P}_{k-1}\mathbf{\Theta}_{k-1} ^{-1}\mathbf{P}_{k-1}^T\mathbf{A}^T\mathbf{s}_k\mathbf{s}_k^T}{\|\mathbf{s}_k\|_2^2-\|\mathbf{d}_{k-1}\|_2^2}$. It is easy to verify that $(\mathbf{I}-P)^2=\mathbf{I}-P$.\\
It follows from the above relation that
$$
\begin{aligned}
\|\mathbf{A}\overline{\mathbf{x}}_k\|_2^2&=\|(\mathbf{I}-P)\mathbf{A}\overline{\mathbf{x}}_{k-1}\|_2^2\\
&=\|\mathbf{A}\overline{\mathbf{x}}_{k-1}\|_2^2-\|P\mathbf{A}\overline{\mathbf{x}}_{k-1}\|_2^2\\
&=\|\mathbf{A}\overline{\mathbf{x}}_{k-1}\|_2^2-\left\|\frac{(\mathbf{I}-\mathbf{A}\mathbf{P}_{k-1}\mathbf{\Theta}_{k-1} ^{-1}\mathbf{P}_{k-1}^T\mathbf{A}^T)\mathbf{s}_k\mathbf{s}_k^T}{\|\mathbf{s}_{k}\|_2^2-\|\mathbf{d}_{k-1}\|_2^2}\mathbf{A}\overline{\mathbf{x}}_{k-1}\right\|_2^2\\
&=\|\mathbf{A}\overline{\mathbf{x}}_{k-1}\|_2^2-\frac{\|(\mathbf{I}-\mathbf{A}\mathbf{P}_{k-1}\mathbf{\Theta}_{k-1} ^{-1}\mathbf{P}_{k-1}^T\mathbf{A}^T)\mathbf{s}_k\mathbf{s}_k^T\mathbf{A}\overline{\mathbf{x}}_{k-1}\|_2^2}{(\|\mathbf{s}_{k}\|_2^2-\|\mathbf{d}_{k-1}\|_2^2)^2}\\
&=\|\mathbf{A}\overline{\mathbf{x}}_{k-1}\|_2^2-\frac{(\|\mathbf{s}_{k}\|_2^2-\|\mathbf{d}_{k-1}\|_2^2)|\mathbf{s}_{k}^T\mathbf{A}\overline{\mathbf{x}}_{k-1}|^2}{(\|\mathbf{s}_{k}\|_2^2-\|\mathbf{d}_{k-1}\|_2^2)^2}\\
&=\|\mathbf{A}\overline{\mathbf{x}}_{k-1}\|_2^2-\frac{|\mathbf{s}_{k}^T\mathbf{A}\overline{\mathbf{x}}_{k-1}|^2}{\|(\mathbf{I}-\mathbf{A}\mathbf{P}_{k-1}\mathbf{\Theta}_{k-1} ^{-1}\mathbf{P}_{k-1}^T\mathbf{A}^T)\mathbf{s}_{k}\|_2^2}\\
&=\|\mathbf{A}\overline{\mathbf{x}}_{k-1}\|_2^2-\frac{|\mathbf{s}_{k}^T\mathbf{A}\overline{\mathbf{x}}_{k-1}|^2}{\sin^2(\vartheta)\|\mathbf{s}_{k}\|_2^2}\\
&=\|\mathbf{A}\overline{\mathbf{x}}_{k-1}\|_2^2-\frac{\|\mathbf{A}_{:,\tau_k}^T\mathbf{A}\overline{\mathbf{x}}_{k-1}\|_2^4}{\sin^2(\vartheta)\|\mathbf{s}_{k}\|_2^2},\\
\end{aligned}
$$
where $\vartheta $ is the angle between the vectors  $\mathbf{s}_{k}$ and $\mathbf{A}\mathbf{P}_{k-1}\mathbf{\Theta}_{k-1} ^{-1}\mathbf{P}_{k-1}^T\mathbf{A}^T$. On the other hand, it follows that
\begin{equation}\label{3-4}
\begin{aligned}
\mathbb{E}_{\tau_k}[\|\mathbf{s}_k\|_2^2|\mathcal{F}_{k-1}] &=\mathbb{E}_{\tau_k}[\|\mathbf{A}_{:,\tau_k}\mathbf{A}_{:,\tau_k}^T\mathbf{A}\overline{\mathbf{x}}_{k-1}\|_2^2|\mathcal{F}_{k-1}]\\
&= \mathbb{E}_{\tau_k}\left[ \left\| \sum_{i \in \tau_k} (\mathbf{A}_{(i)}^T \mathbf{A}\overline{\mathbf{x}}_{k-1}) \mathbf{A}_{(i)} \right\|_2^2 \,\bigg|\, \mathcal{F}_{k-1} \right] \\
&= p_{ij}(i \in \tau_k, j \in \tau_k)\sum_{i=1}^n \sum_{j=1}^n (\mathbf{A}_{(i)}^T\mathbf{A}\overline{\mathbf{x}}_{k-1})(\mathbf{A}_{(j)}^T\mathbf{A}\overline{\mathbf{x}}_{k-1}) \mathbf{A}_{(i)}^T \mathbf{A}_{(j)}  \\
&=p_{i}(i \in \tau_k)\sum_{i=1}^m(\mathbf{A}_{(i)}^T \mathbf{A}\overline{\mathbf{x}}_{k-1})^2\|\mathbf{A}_{(i)}\|_2^2 \\
&\quad+ p_{ij}(i,j \in \tau_k)\sum_{i\neq j}(\mathbf{A}_{(i)}^T\mathbf{A}\overline{\mathbf{x}}_{k-1})(\mathbf{A}_{(j)}^T\mathbf{A}\overline{\mathbf{x}}_{k-1}) \mathbf{A}_{(i)}^T \mathbf{A}_{(j)} \\
&=l^2\|\mathbf{A}\mathbf{A}^T\mathbf{A}\overline{\mathbf{x}}_{k-1}\|_2^2+l(1-l)\lambda_{max}(\mathbf{A}^T\mathbf{A})\|\mathbf{A}^T\mathbf{A}\overline{\mathbf{x}}_{k-1}\|_2^2\\
&\leq l\lambda_{max}^2(\mathbf{A}^T\mathbf{A})\|\mathbf{A}\overline{\mathbf{x}}_{k-1}\|_2^2.
\end{aligned}
\end{equation}
The result follows from taking the conditional expectation of block $\tau_k$ over the entire history $\mathcal{F}_{k-1} = \{\Gamma_1,\cdots, \Gamma_{k-1}\}$ and \eqref{3-4}, which yields
$$
\begin{aligned}
\mathbb{E}_{\tau_k}[\|\mathbf{A}\overline{\mathbf{x}}_k\|_2^2|\mathcal{F}_{k-1}]&\leq \|\mathbf{A}\overline{\mathbf{x}}_{k-1}\|_2^2-\mathbb{E}_{\tau_k}\left[\frac{\|\mathbf{A}_{:,\tau_k}^T\mathbf{A}\overline{\mathbf{x}}_{k-1}\|_2^4}{\sin^2(\vartheta)\|\mathbf{s}_{k}\|_2^2}\Bigg| \mathcal{F}_{k-1}\right]\\
&\leq\|\mathbf{A}\overline{\mathbf{x}}_{k-1}\|_2^2-\frac{\mathbb{E}_{\tau_k}[\|\mathbf{A}_{:,\tau_k}^T\mathbf{A}\overline{\mathbf{x}}_{k-1}\|_2^4|\mathcal{F}_{k-1}]}{\sin^2(\vartheta)\mathbb{E}_{\tau_k}[\|\mathbf{s}_{k}\|_2^2|\mathcal{F}_{k-1}]}\\
&=\|\mathbf{A}\overline{\mathbf{x}}_{k-1}\|_2^2-\frac{(\sum_{i=1}^n p_i|r^{(i)}_{k-1}|^2)^2}{\sin^2(\vartheta)\mathbb{E}_{\tau_k}[\|\mathbf{s}_{k}\|_2^2|\mathcal{F}_{k-1}]}\\
&\leq\|\mathbf{A}\overline{\mathbf{x}}_{k-1}\|_2^2-\frac{l^2\lambda_{min}^2(\mathbf{A}^T\mathbf{A})\|\mathbf{A}\overline{\mathbf{x}}_{k-1}\|_2^4}{\sin^2(\vartheta)l\lambda_{max}^2(\mathbf{A}^T\mathbf{A})\|\mathbf{A}\overline{\mathbf{x}}_{k-1}\|_2^2}\\
&=\left(1-\frac{l\lambda_{min}^2(\mathbf{A}^T\mathbf{A})}{\sin^2(\vartheta)\lambda_{max}^2(\mathbf{A}^T\mathbf{A})}\right)\|\mathbf{A}\overline{\mathbf{x}}_{k-1}\|_2^2.
\end{aligned}
$$
Therefore, by taking the expectation over the entire history, the
proof follows immediately.


\end{proof}
\begin{rem}\label{Rem-7}
\rm The convergence factor generated by Theorem \ref{The-3-3} still satisfies $0\leq1-\frac{l\lambda_{min}^2(\mathbf{A}^T\mathbf{A})}{\sin^2(\vartheta)\lambda_{max}^2(\mathbf{A}^T\mathbf{A})}<1$, and the proof is similar to Lemma \ref{lem-5}, hence we omit it. Moreover, if the parameter $l=1$, then Algorithm \ref{algo-2-5}  reduces to the following PLSS residual least squares method.
\end{rem}

\begin{algorithm}
	\caption{RPLSS-LS}
	\label{algo-3-4}
	\begin{algorithmic}[1]
        \REQUIRE $\mathbf{A}, \mathbf{b},\mathbf{W}=(\mathbf{A}^T\mathbf{A})^{-1}$\\
        \ENSURE $\mathbf{x}_{k}$\\
        \STATE
       For $k=0, 1, 2,\cdots $, do:\\
 \STATE
   Compute 
   $$\mathbf{s}_k=\mathbf{A}\mathbf{r}_{k-1}$$
   where $\mathbf{r}_{k-1}=\mathbf{A}^T(\mathbf{b}-\mathbf{A}\mathbf{x}_{k-1})$
    \STATE
   Compute 
   $$\mathbf{p}_k=\gamma_{k-1}(\mathbf{W}\mathbf{A}^T\mathbf{s}_k+\frac{\|\mathbf{r}_{k-1}\|_2^2}{\|\mathbf{A}\mathbf{p}_{k-1}\|_2^2}\mathbf{p}_{k-1})$$
   where $\gamma_{k-1}=\frac{\mathbf{s}_k^T(\mathbf{b}-\mathbf{A}\mathbf{x}_{k-1})}{\mathbf{s}_k^T\mathbf{A}\mathbf{W}\mathbf{A}^T\mathbf{s}_k-\|\mathbf{d}_{k-1}\|_2^2}$ with $\mathbf{d}_{k-1}=\mathbf{\Theta}_{k-1} ^{-1/2}\|
   \mathbf{r}_{k-1}\|_2^2\mathbf{e}_{k-1}$ and $\mathbf{\Theta}_{k-1}=\mathbf{P}_{k-1}^T\mathbf{W}^{-1}\mathbf{P}_{k-1}$.
            \STATE
            Set $\mathbf{x}_{k} =\mathbf{x}_{k-1}+\mathbf{p}_k$.

	\end{algorithmic}  
\end{algorithm}

\section{Numerical Experiments}\label{section-4}

In this section, we record some numerical results to illustrate the feasibility and validity of the proposed methods compared with some state-of-the-art randomized solvers for consistent and inconsistent linear systems \eqref{1-1}. We first test the methods with randomly generated coefficient matrices or those coming from the University of Florida Sparse Matrix Collection. Then we present additional numerical results on the X-ray computed tomography problem to further illustrate the effectiveness of the proposed method in comparison with the solvers mentioned above. Denote by ``IT", the number of iterations, and by ``CPU", the elapsed computing time in seconds. For each method, we report the mean computing time in seconds and the mean number of iterations based on their average values of $5$ repeated tests. Additionally, all experiments have been carried out in MATLAB
2021b on a personal computer with Inter(R) Core(TM) i5-13500HX @2.5GHz and 32.00 GB memory. In Algorithms \ref{algo-2-1}, \ref{algo-2-2} and \ref{algo-2-5}, the parameter $l$ is set to $0.5$. The symbol $\dagger$ indicates that the solver fails to converge after reaching 1000 seconds. 
All test methods are as follows:\\
$\star$ GRK: Greedy randomized Kaczmarz method \cite{Bai}.\\
$\star$ RaBK: Randomized average block Kaczmarz method \cite{Necoara}.\\
$\star$ FGBK: Fast greedy block Kaczmarz method \cite{ZK}.\\
$\star$ RBK(50): Randomized block Kaczmarz method \cite{Needell2} with the block size set to be 50.\\
$\star$ REABK: Randomized extended average block Kaczmarz method\cite{KDu}.\\
$\star$ GBEK: Greedy block extended Kaczmarz method\cite{Ke}.\\
$\star$ RCD: Randomized coordinate descent method \cite{LeD}.\\
$\star$ GRCD: Greedy randomized coordinate descent method\cite{Bai4}.\\

\begin{example}\label{Example-1}
\rm Consider solving the Eq.\eqref{1-1} with its coefficient matrix given by the MATLAB function $sprandn$. Moreover, we randomly generate a solution $\mathbf{x}_\star$ by the MATLAB function $randn$, with $\mathbf{x}_\star=randn(n,1)$. For consistent linear systems, the right-hand side vector $\mathbf{b}$ is constructed as $\mathbf{b}=\mathbf{A}\mathbf{x}_\star$. For inconsistent linear systems, we set $\mathbf{b}=\mathbf{A}\mathbf{x}_\star+r$, where $r$ is a nonzero vector belonging to the null space of $\mathbf{A}^T$. The stopping criteria of consistent linear systems is the error norm at the current iterate $\mathbf{x}_{k}$ satisfying 
$$ \|\mathbf{x}_{k}-\mathbf{x}_{\star} \| _{2}<10^{-6},$$ and the stopping criteria of inconsistent linear systems is the relative error norm at the current iterate $\mathbf{x}_{k}$ satisfying 
$$\mathrm{RSE}:= \frac{\|\mathbf{x}_k-\mathbf{x}_{\star}\|_2^2}{\|\mathbf{x}_{\star}\|_2^2}< 10^{-6}.$$
For the RaBK solver, we set $\tau=m/\|\mathbf{A}\|_2^2$ and $\alpha=1.95L_k$ with $$L_k = \frac{\sum_{i\in \tau_k} \bar{\omega}_i^k (
(\mathbf{A}^{(i)})^T \mathbf{x}^k - \mathbf{b}^{(i)})^2}{\left\| \sum_{i\in \tau_k} \bar{\omega}_i^k(
(\mathbf{A}^{(i)})^T \mathbf{x}^k - \mathbf{b}^{(i)}) \mathbf{A}^{(i)} \right\|^2}.$$

We implement all tested methods to solve consistent linear systems \eqref{1-1}, and then record the obtained numerical results in Tables \ref{Table-1} and \ref{Table-2}. A rough conclusion drawn from these tables is that, RPLSS-GK significantly outperforms GRK, RaBK, FGBK, and RBK(50) in all tests of consistent linear systems (overdetermined, underdetermined, and of different scales), making it the most competitive algorithm in the current table. The computation time of RPLSS-GK is about $2-3$ times faster than FGBK, and $1-2$ orders of magnitude faster than RaBK and RBK(50). Especially in ultra-large-scale problems with millions of rows, RPLSS-GK only takes $3-20$ seconds, while other methods take tens to hundreds of seconds. This demonstrates RPLSS-GK is more favorable for big data problems. {The convergence curves of all methods with $300000\times 40000$ and $50000\times 100000$ were displayed in Fig.\ref{Fig:1}. It can be clearly seen from the figure that RPLSS-GK converges faster than other existing methods.}

\begin{small}
\begin{table}[!htbp]
    \centering
    \caption{Numerical results  of Example \ref{Example-1} when Eq.\eqref{1-1} is consistent  \label{Table-1}}
    \setlength{\tabcolsep}{0.7mm}
    \begin{tabular}{|l|c|c|c|c|c|}
    \hline
        \multicolumn{2}{|c|}{$ m \times n$}  & $20000\times5000$ & $300000\times40000$ & $600000\times50000$ & $1000000\times100000$ \\ \hline
         & IT   & 7896 & $\dagger$ & $\dagger$ & $\dagger$ \\ 
         GRK & CPU  & 16.0886 & $\dagger$ & $\dagger$ & $\dagger$\\ \hline

         & IT  & 1022 & 1082 & 846 & 958 \\ 
       RaBK& CPU & 0.3399 & 11.0386 & 17.0775 & 20.9910\\ \hline

         & IT  & 31 & 84 & 60 & 53\\ 
        FGBK & CPU & 0.0929 & 2.7069 & 6.0688 & 9.6621 \\ \hline

         & IT  & 73 & 286 & 155 & 38\\ 
        RBK(50)& CPU  & 0.1363 & 7.4937 & 5.7692 & 30.8101 \\  \hline

        & IT  & 17 & 29 & 22 & 18  \\ 
       RPLSS-GK  & CPU  & \textbf{0.0322} & \textbf{0.8366} & \textbf{2.3663} & \textbf{3.4091}\\  \hline
    \end{tabular}
\end{table}
\end{small}

\begin{small}
\begin{table}[!htbp]
    \centering
    \caption{Numerical results  of Example \ref{Example-1} when Eq.\eqref{1-1} is consistent  \label{Table-2}}
    \setlength{\tabcolsep}{0.7mm}
    \begin{tabular}{|l|c|c|c|c|c|}
    \hline
        \multicolumn{2}{|c|}{$ m \times n$}  & $5000\times10000$ & $50000\times100000$ & $100000\times250000$ & $200000\times500000$ \\ \hline
         & IT   & 269100 & $\dagger$ & $\dagger$ & $\dagger$ \\ 
         GRK & CPU  & 813.1132 & $\dagger$ & $\dagger$ & $\dagger$\\ \hline

         & IT  & 1119 & 1178 & 1211 & 1212 \\ 
       RaBK& CPU & 4.2261 & 22.0106 & 98.5854 & 375.4899\\ \hline

         & IT  & 435 & 641 & 504 & 254 \\ 
        FGBK & CPU & 0.7145 & 5.2255 & 19.7438 & 43.5896 \\ \hline

         & IT  & 8540 & 9081 & 5965 & 5973 \\ 
        RBK(50)& CPU  & 30.7323 & 160.3646 & 250.6066 & 446.2368 \\  \hline

        & IT  & 127 & 143 & 112 & 112 \\ 
       RPLSS-GK  & CPU  & \textbf{0.2353} & \textbf{1.7733} & \textbf{5.0777} & \textbf{20.3680}\\  \hline
    \end{tabular}
\end{table}
\end{small}

\begin{figure}[htbp]
\begin{center}
	\begin{minipage}[c]{1\textwidth}
		\includegraphics[width=2.4in]{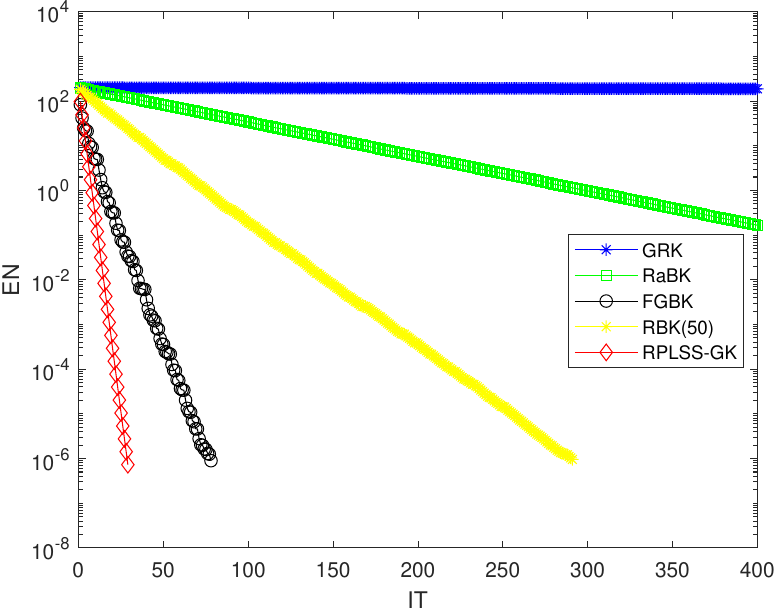} 
		\includegraphics[width=2.4in]{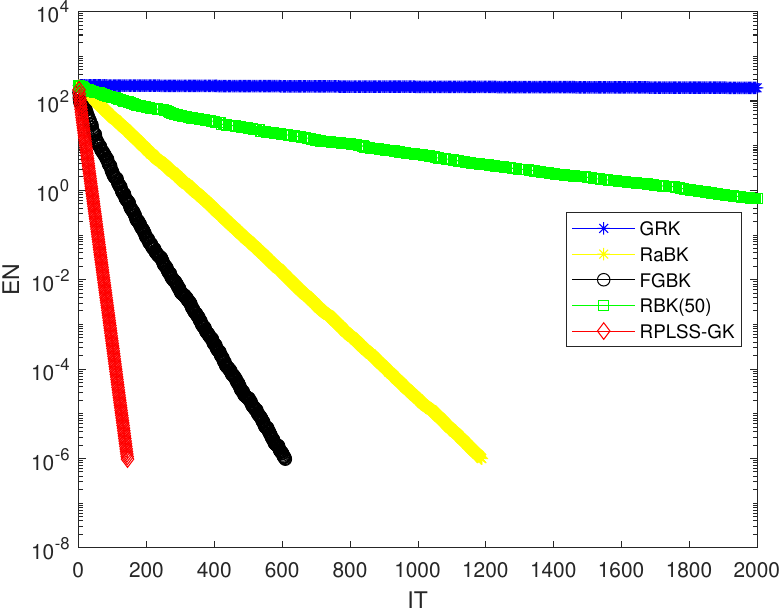}
	\end{minipage}
	\caption{ {Convergence curves of Example \ref{Example-1} with the consistent cases $300000\times 40000$ (left) and $50000\times 100000$ (right).}}\label{Fig:1}
 \end{center}
\end{figure}

For the inconsistent case, we report the numerical results regarding the iterations and computing time in Table \ref{Table-3}. As shown in the table, RPLSS-GCD and RPLSS-GEK converge stably in all test scales, and the number of iterations and computing time are much better than other methods, especially showing great application potential in very large-scale problems. Among traditional methods, only the GBEK method converges on all scale matrices, the REABK algorithm fails on the largest scale case, and RCD and GRCD also fail at three and two cases, respectively. {Moreover, we displayed
the convergence curves of all methods with the case $m=10000$ and $n=2000$ in Fig.\ref{Fig:2}, which coincide with the above observations.}

\begin{small}
\begin{table}[!htbp]
    \centering
    \caption{Numerical results  of Example \ref{Example-1} when Eq.\eqref{1-1} is inconsistent  \label{Table-3}}
    \setlength{\tabcolsep}{0.4mm}
    \begin{tabular}{|l|c|c|c|c|c|}
    \hline
        \multicolumn{2}{|c|}{$ m \times n$}  & $10000\times2000$ & $50000\times20000$ & $200000\times80000$ & $1000000\times100000$ \\ \hline
         & IT   & 5855 &29953& 78568 & $\dagger$ \\ 
         REABK & CPU  & 2.1651 & 34.0505& 232.9871& $\dagger$\\ \hline

         & IT  & 274& 5920 & 818  & 175 \\ 
       GBEK& CPU & 0.1826&  20.8799& 34.1692 & 51.2237\\ \hline

         & IT  & 54282 &$\dagger$ & $\dagger$ & $\dagger$\\ 
        RCD & CPU &  14.4246&  $\dagger$& $\dagger$&$\dagger$ \\ \hline

         & IT  & 12708& 280691 & $\dagger$ & $\dagger$\\ 
        GRCD& CPU  & 3.8332 &  790.5998& $\dagger$ & $\dagger$ \\  \hline

        & IT  & 2000 &$\dagger$ & $\dagger$& $\dagger$  \\ 
       RPLSS-CD(10)  & CPU  & 3.7865 &$\dagger$ &$\dagger$& $\dagger$\\  \hline

        & IT & 30 & 54 & 41 & 23\\ 
        RPLSS-GCD & CPU  & \textbf{0.0193} &  \textbf{0.1991}& \textbf{1.4567} & \textbf{6.6872} \\ \hline
        
         & IT  & 40 &113& 65&  25  \\ 
        RPLSS-GEK& CPU  & 0.0457 &  0.6718 & 2.6305 & 9.4435 \\  \hline

    \end{tabular}
\end{table}
\end{small}

\begin{figure}[!htbp]
	\begin{center}
		\begin{minipage}[c]{0.6\textwidth}
			\includegraphics[width=2.8in]{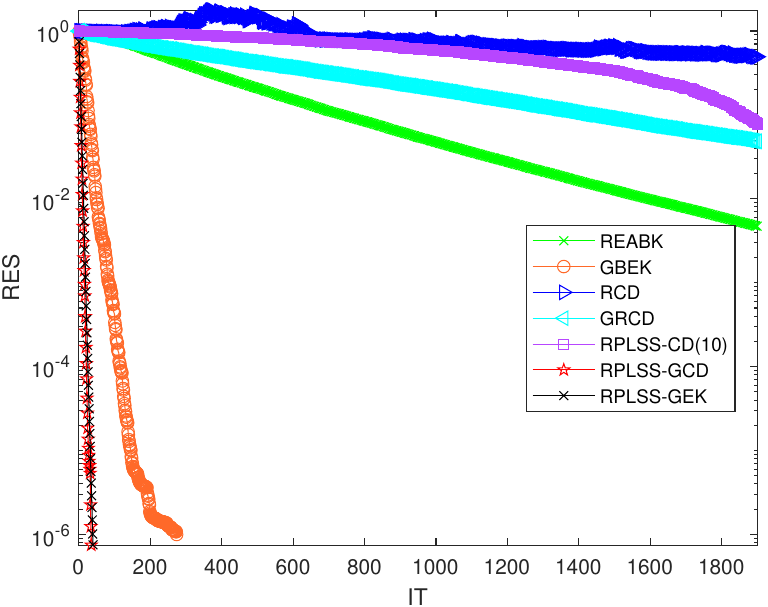}
		\end{minipage}\hspace{3.2em}
	\end{center}
	\vspace{0em}\caption{{Convergence curves of Example \ref{Example-1} with the inconsistent case $10000\times 2000$.}\label{Fig:2}}
\end{figure}

\end{example}

\begin{example}\label{Example-2}
\rm Consider solving the consistent linear system \eqref{1-1} with its coefficient matrix coming from the University of Florida Sparse Matrix Collection \cite{Davis}. The chosen matrices are mainly used for statistical, mathematical, combinatorial, and linear programming problems. Moreover, we randomly generate a solution $\mathbf{x}_\star$ by the MATLAB function $randn$, with $\mathbf{x}_\star=randn(n,1)$, such that the right-hand side vector $\mathbf{b}$ is set to be $\mathbf{A}\mathbf{x}_\star$. The stopping criteria of all tested methods are the error norm at the current iterate $\mathbf{x}_{k}$ satisfying 
$$ \|\mathbf{x}_{k}-\mathbf{x}_{\star} \| _{2}<10^{-6}.$$

\begin{small}
\begin{table}[!htbp]
    \centering
    \caption{Numerical results  of Example \ref{Example-2}  \label{Table-4}}
    \setlength{\tabcolsep}{0.3mm}
    \begin{tabular}{|l|c|c|c|c|c|}
    \hline
        \multicolumn{2}{|c|}{Matrix} & Franz11 & GL7d14 & relat8 & testbig  \\ \hline
        \multicolumn{2}{|c|}{$ m \times n$} & $47104\times30144$ & $171375\times47271$ & $345688\times12347$ & $17613\times31223$ \\ \hline
        \multicolumn{2}{|c|}{Density}  & 0.02\% & 0.02\% & 0.03\% & 0.01\%  \\ \hline

        GRK & IT & $\dagger$ & $\dagger$  & $\dagger$  & 1113996  \\ 
        ~ & CPU & $\dagger$ & $\dagger$ & $\dagger$ & 971.3152 \\  \hline

        RaBK & IT &  20122 & 98756  & 87272   & $\dagger$ \\ 
        ~ & CPU & 11.4521 & 87.3713 & 38.6179& $\dagger$ \\  \hline

        FGBK & IT & 484 & 1541  & 1766 & 166061 \\ 
        ~ & CPU & 0.6100 & 9.2269 & 12.3514 & 81.9880 \\  \hline

        RBK(50) & IT & 1912 & 4083  & 1453 & 228235\\ 
        ~ & CPU & 11.8855 & 70.2431 & 16.8130 & 834.0728 \\  \hline

        RPLSS-GK & IT & 81 &98& 107 & 50\\ 
        ~ & CPU & \textbf{0.2017} & \textbf{0.9340} & \textbf{1.3551} & \textbf{0.7952} \\ \hline

    \end{tabular}
\end{table}
\end{small}

\begin{small}
\begin{table}[!htbp]
    \centering
    \caption{Numerical results  of Example \ref{Example-2}  \label{Table-5}}
    \setlength{\tabcolsep}{0.1mm}
    \begin{tabular}{|l|c|c|c|c|c|}
    \hline
        \multicolumn{2}{|c|}{Matrix} &sls & kneser\_10\_4\_1 & degme & ins2  \\ \hline
        \multicolumn{2}{|c|}{$ m \times n$} & $1748122\times62729$ & $349651\times330751$ & $185501\times659415$ & $309412\times309412$ \\ \hline
        \multicolumn{2}{|c|}{Density}  & 0.006\% & 0.0009\% & 0.007\% & 0.003\%  \\ \hline
        GRK & IT & $\dagger$ & $\dagger$  & $\dagger$  & $\dagger$ \\ 
        ~ & CPU & $\dagger$ & $\dagger$ & $\dagger$ & $\dagger$ \\  \hline

        RaBK & IT & $\dagger$ & $\dagger$  & 800  & $\dagger$ \\ 
        ~ & CPU & $\dagger$ & $\dagger$ & 45.0287& $\dagger$ \\  \hline

        FGBK & IT & $\dagger$ & 14914  & 2622 & $\dagger$ \\ 
        ~ & CPU & $\dagger$ & 200.3864 & 43.8835 & $\dagger$ \\  \hline

        RBK(50) & IT & 429 & $\dagger$  & 2554 & $\dagger$\\ 
        ~ & CPU & 24.2707 & $\dagger$ & 160.1261 & $\dagger$ \\  \hline

        RPLSS-GK & IT & 290 &804& 103 & 381\\ 
        ~ & CPU & \textbf{14.2922} & \textbf{39.6706} & \textbf{3.2724} & \textbf{11.6757} \\ \hline

    \end{tabular}
\end{table}
\end{small}

\begin{figure}[htbp]
\begin{center}
	\begin{minipage}[c]{0.95\textwidth}
		\includegraphics[width=2.4in]{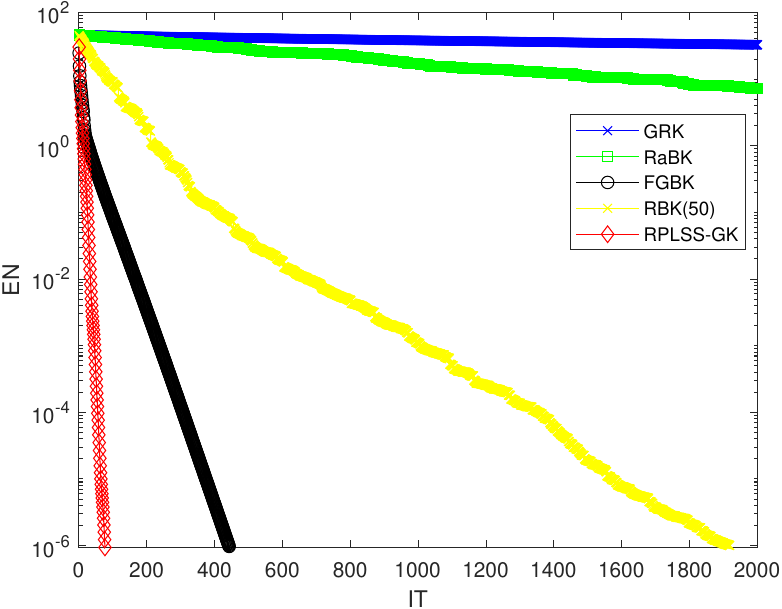}\
		\includegraphics[width=2.4in]{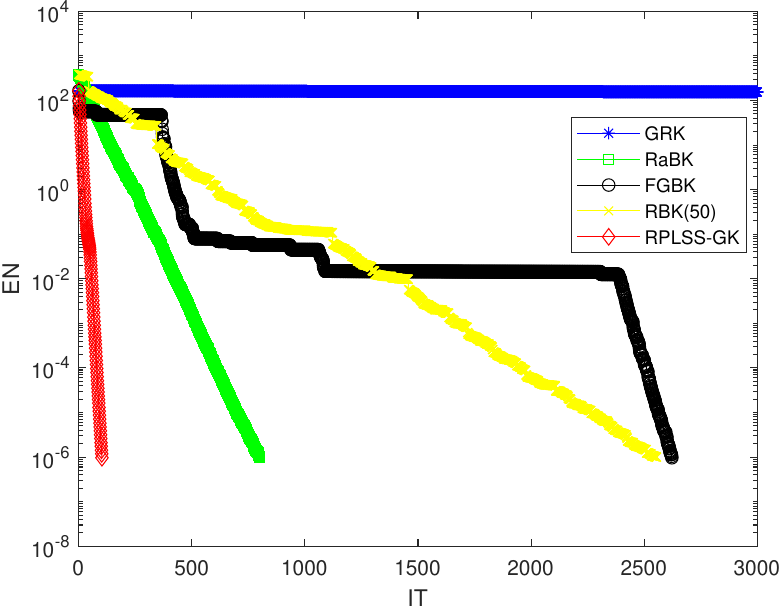}
	\end{minipage}
	\caption{{Convergence curves of Example \ref{Example-2} with the cases Franz11 (left) and degme (right).}}\label{Fig:3}
 \end{center}
\end{figure}

We run the test methods on some practical problems, and then list the numerical results in Tables \ref{Table-4} and \ref{Table-5}. As shown in these tables, GRK fails to converge in most cases. Although RaBK, FGBK, and RBK converge on some problems, the number of iterations and computing time required are much higher than those of RPLSS-GK. Traditional Kaczmarz-type methods generally fail on very large-scale matrices (sls, ins2) and some singular matrices, while RPLSS-GK converges stably on all test matrices, showing stronger robustness. A rough conclusion is that the RPLSS-GK method significantly outperforms the other solvers in terms of the number of iterations and CPU time, especially for ill-conditioned and large-scale sparse matrices, where its advantages are even more pronounced. Particularly, the speed-up of RPLSS-GK compared with RaBK is at least $13.71$ in the case of degme, and at most $93.54$ in the case of GL7d14, compared with FGBK is at least $3.02$ in the case of Franz11, and at most $103.10$ in the case of testbig, and compared with RBK is at least $1.69$ in the case of sls, and at most $1048.88$ in the case of testbig.  {Fig.\ref{Fig:3} plots the convergence curves of all tested methods under the Franz11 and degme test cases. It is evident that the proposed RPLSS-GK solver achieves superior convergence performance compared with the other randomized solvers, consistent with the aforementioned observations.}
\end{example}

\begin{example}\label{Example-3}
\rm Consider solving the inconsistent linear system \eqref{1-1} with its coefficient matrix coming from the University of Florida Sparse Matrix Collection \cite{Davis}. The chosen matrices are mainly used for combinatorial and linear programming problems. 
Moreover, we randomly generate a solution $\mathbf{x}_\star$ by the MATLAB function $randn$, with $\mathbf{x}_\star=randn(n,1)$, such that the right-hand side vector $\mathbf{b}$ is set to be $\mathbf{A}\mathbf{x}_\star+r$, where $r\in $ null$(\mathbf{A}^T)$. The stopping criteria of all tested methods are the relative error norm at the current iterate $\mathbf{x}_{k}$ satisfying 
$$\mathrm{RSE}:= \frac{\|\mathbf{x}_k-\mathbf{x}_{\star}\|_2^2}{\|\mathbf{x}_{\star}\|_2^2}< 10^{-6}.$$


\begin{table}[!htbp]
    \centering
    \caption{Numerical results  of Example \ref{Example-3}   \label{Table-6}}
    \setlength{\tabcolsep}{1.6mm}
    \begin{tabular}{|l|c|c|c|c|c|c|}
    \hline
        \multicolumn{2}{|c|}{Matrix} & ch7-9-b5 & D6-6&n4c6-b7 &cage12& abtaha2\\ \hline
        \multicolumn{2}{|c|}{$ m $} & $423360$ & $120576$ & $ 163215$ & $130228$ &$37932$\\ \hline
        \multicolumn{2}{|c|}{$n$} & $317520$ & $23740$ & $ 104115$ & $130228$ &$331$\\ \hline
        \multicolumn{2}{|c|}{Density}  & 0.12\% & 0.36\% & 0.01\% & 0.01\%  &1.09\%\\ \hline
        REABK & IT & 39233 & 40961&6778 & 737472&  27128\\ 
        ~ & CPU & 353.2802 & 61.2662 &  13.6403 &  956.9433&  12.8866 \\ \hline
        GBEK & IT & 359 &871 &117 &547& 925\\ 
        ~ & CPU & 6.4408 & 2.8610 & 0.9547 &  4.6600&  1.0560 \\ \hline

        RPLSS-GEK & IT & 33 & 113&19 & 160& 285 \\ 
        ~ & CPU & \textbf{0.9433}& \textbf{0.6353} & \textbf{0.2159} & \textbf{1.5078}& \textbf{0.5514}\\ \hline

    \end{tabular}
\end{table}

\begin{figure}[htbp]
\begin{center}
	\begin{minipage}[c]{0.95\textwidth}
		\includegraphics[width=2.4in]{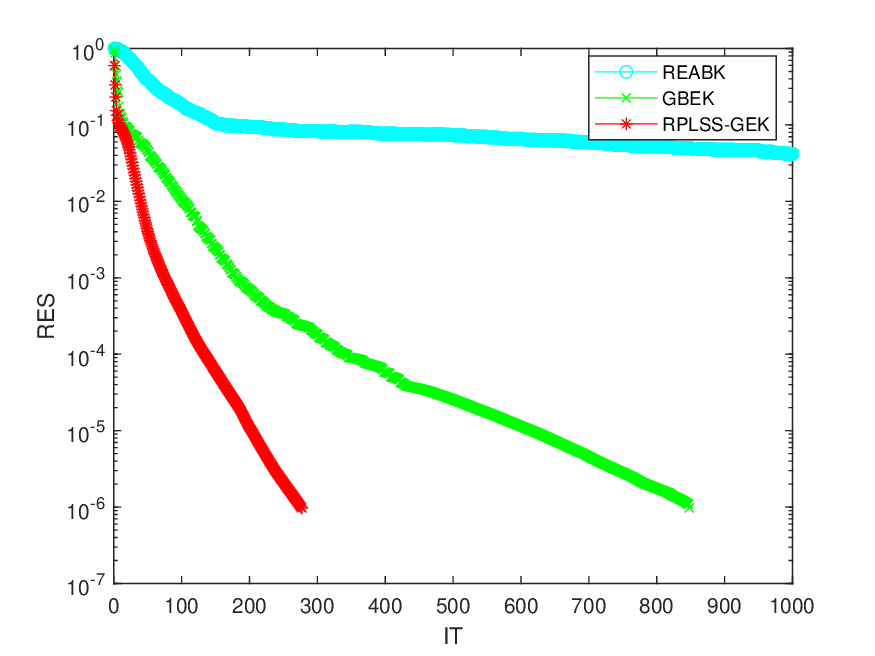}\
		\includegraphics[width=2.44in]{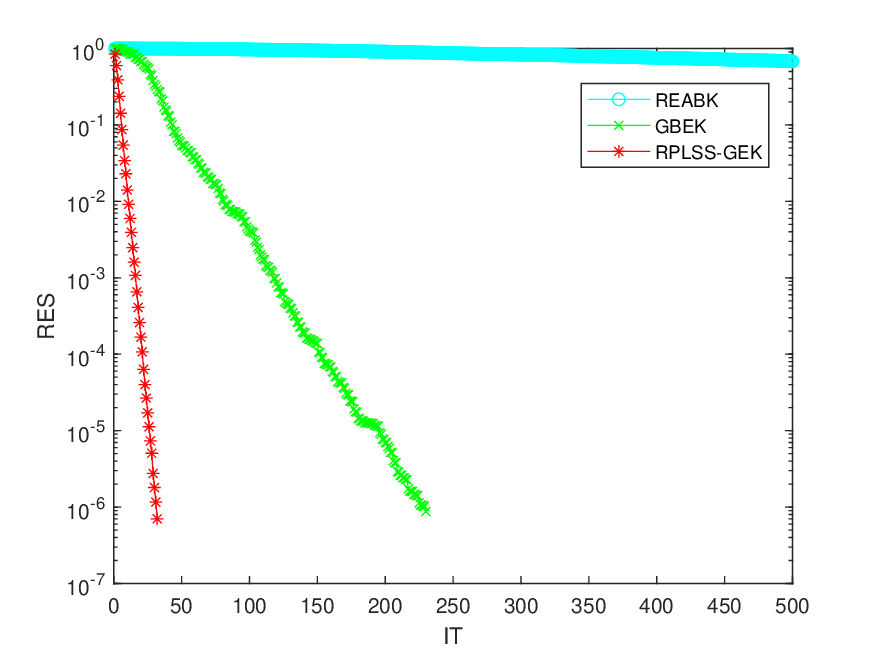}
	\end{minipage}
	\caption{{Convergence curves of Example \ref{Example-3} with the cases abtaba2 (left) and lp\_nug20 (right).}}\label{Fig:4}
 \end{center}
\end{figure}

In Tables \ref{Table-6}-\ref{Table-7}, we report the numerical results regarding the iterations and computing time. As shown in these tables, one can reasonably conclude that the RPLSS-GEK solver outperforms all other tested solvers. 
RPLSS-GEK converges stably on all test matrices, with significantly fewer iterations and shorter computation time compared to existing methods like REABK and GBEK. It demonstrates excellent robustness and efficiency, especially for highly ill-conditioned matrices or ultra-large-scale sparse matrices (such as Delor64K). More precisely, compared with the REABK solver, RPLSS-GEK brings an amazing speed-up of $23.37$ times to $634.66$ times. Compared with the efficient GBEK solver, RPLSS-GEK also maintains a stable acceleration of $1.91$ to $8.63$ times.  {Additionally, we depicted the convergence curves of all tested methods with the cases abtaba2 and lp\_nug20 in Fig.\ref{Fig:4}. As seen, the RPLSS-GEK solver still work more effectively than the REABK and GBEK solvers, respectively, which coincides with the abovementioned conclusion.}


\begin{small}
\begin{table}[!htbp]
    \centering
    \caption{Numerical results  of Example \ref{Example-3}   \label{Table-7}}
    \setlength{\tabcolsep}{1.6mm}
    \begin{tabular}{|l|c|c|c|c|c|c|}
    \hline
        \multicolumn{2}{|c|}{Matrix} & lp\_nug20 & bibd\_49\_3&GL7d24 &rlfdual& Delor64K\\ \hline
        \multicolumn{2}{|c|}{$ m $} & $15240$ & $1176$ & $ 21074$ & $8052$ &$64719$\\ \hline
        \multicolumn{2}{|c|}{$n$} & $72600$ & $18424$ & $ 105054$ & $74970$ &$1785345$\\ \hline
        \multicolumn{2}{|c|}{Density}  & 0.03\% & 2.55\% & 0.03\% & 0.05\%  &0.0006\%\\ \hline
        REABK & IT & 7632 & 1433 &135167 & 67152& $\dagger$\\ 
        ~ & CPU & 11.6339 & 0.6571 & 130.6364 &  130.6364&  $\dagger$ \\ \hline
        GBEK & IT & 211 & 52 &8072 &1319& $\dagger$\\ 
        ~ & CPU & 0.4528 & 0.0586 & 32.0111 &  3.3230&  $\dagger$ \\ \hline

        RPLSS-GEK & IT & 32 & 17 &1543& 722& 341 \\ 
        ~ & CPU & \textbf{0.1067}& \textbf{0.0197} & \textbf{3.7095}& \textbf{1.2273}& \textbf{9.7724}\\ \hline

    \end{tabular}
\end{table}
\end{small}

\end{example}

\begin{example}\label{Example-4}
\rm Consider solving the problem of X-ray computed tomography, which is one of the classic application areas of the Kaczmarz method. The Matlab package AIR Tools II \cite{Hansen} is used to simulate the above problems, so as to verify the effectiveness and performance of the proposed method for different linear systems. In this section, two sets of comparative experiments are designed, corresponding to consistent and inconsistent linear systems, respectively. First, we generated a test image of size 256 by 256 based on the $paralleltomo$ function in the AIR Tools II package and recorded the measurements of the angles 0:0.5:179. The number of sensors per projection line is 280. The right-hand side vector is
$\mathbf{b}=\mathbf{A}\mathbf{x}_{\star}$, where the unique solution of the consistent linear systems \eqref{1-1} is obtained by reshaping the $256\times256$ Shepp-Logan medical phantom. We verified the effectiveness of the algorithms under three scenarios: complete projection data, $10\%$ missing projection data, and $30\%$ missing projection data. The stopping criteria of all tested methods are the relative error norm at the current iterate $\mathbf{x}_{k}$ satisfying 
$$\mathrm{RSE}:= \frac{\|\mathbf{x}_k-\mathbf{x}_{\star}\|_2^2}{\|\mathbf{x}_{\star}\|_2^2}< 10^{-2}.$$

In Tables \ref{Table-8}-\ref{Table-10}, we list the number of iteration steps, the computing times, and the relative solution errors for all tested methods. It appears from this table that the RPLSS-GK significantly outperforms GRK, RaBK, FGBK, and RBK(50) in terms of the number of iterations and computation time. Especially for large-scale sparse matrices, its speedup can reach tens or even hundreds of times. As matrix row information gradually becomes missing, the number of iterations and iteration time of other methods increase sharply, while the number of iterations of RPLSS-GK increases slowly and the iteration time remains basically unchanged. This indicates that RPLSS-GK has good adaptability to missing matrix rows and exhibits good robustness. This is because RPLSS-GK constructs a sketch matrix using historical residuals, which has a natural resistance to interference in row sampling, and can still maintain efficient convergence even when $30\%$ of the matrix rows are missing. Therefore, RPLSS-GK is particularly suitable for solving large-scale linear systems with incomplete rows and uneven sampling, which are common in real-world data. The pictures of projection data are displayed in Figs.\ref{Fig:5}-\ref{Fig:7}.

\begin{table}[!htbp]
    \centering
    \caption{Numerical results  of Example \ref{Example-4} for $100520\times65536$ matrix $A$   \label{Table-8}}
    \begin{tabular}{|l|c|c|c|c|}
    \hline
        \multicolumn{2}{|c|}{Algorithm} & IT & CPU &RSE \\ \hline
        \multicolumn{2}{|c|}{GRK} &  20389 & $\dagger$ & 1.91e-2\\ 
        \multicolumn{2}{|c|}{RaBK} & 13439 & 23.1365 & 9.99e-3\\ 
        
        \multicolumn{2}{|c|}{FGBK}  &  123& 5.5835 & 9.99e-3 \\

        \multicolumn{2}{|c|}{RBK(50)}  & 123& 7.0292 & 9.22e-3 \\

        \multicolumn{2}{|c|}{RPLSS-GK}  & 26 & \textbf{1.1416} & 9.58e-3\\ \hline

    \end{tabular}
\end{table}

\begin{table}[!htbp]
    \centering
    \caption{Numerical results  of Example \ref{Example-4} for $90\%$ $100520\times65536$  matrix $A$   \label{Table-9}}
    \begin{tabular}{|l|c|c|c|c|}
    \hline
        \multicolumn{2}{|c|}{Algorithm} & IT & CPU &RSE \\ \hline
        \multicolumn{2}{|c|}{GRK} &  22343  & $\dagger$ & 1.98e-2\\ 
        
        \multicolumn{2}{|c|}{RaBK} &  17002 & 38.2633 & 9.99e-3\\ 
        
        \multicolumn{2}{|c|}{FGBK}  &  184 & 6.9108 & 9.99e-3 \\

        \multicolumn{2}{|c|}{RBK(50)}  & 263& 13.4510 & 9.99e-3 \\

        \multicolumn{2}{|c|}{RPLSS-GK}  &29& \textbf{1.0590} & 9.68e-3\\ \hline

    \end{tabular}
\end{table}

\begin{table}[!htbp]
    \centering
    \caption{Numerical results  of Example \ref{Example-4} for $70\%$ $100520\times65536$ matrix $A$   \label{Table-10}}
    \begin{tabular}{|l|c|c|c|c|}
    \hline
        \multicolumn{2}{|c|}{Algorithm} & IT & CPU &RSE \\ \hline
        \multicolumn{2}{|c|}{GRK} &  27802  & $\dagger$ & 2.27e-2\\ 
        
        \multicolumn{2}{|c|}{RaBK} &24596 & 50.5284 & 9.99e-3\\ 
        
        \multicolumn{2}{|c|}{FGBK}  &  581 & 15.7611 & 9.99e-3 \\

        \multicolumn{2}{|c|}{RBK(50)}  & 681& 28.7061 & 9.99e-3 \\

        \multicolumn{2}{|c|}{RPLSS-GK}  & 42& \textbf{1.2160} & 9.94e-3 \\ \hline

    \end{tabular}
\end{table}

 \begin{figure}[htbp!]
\begin{center}
	\begin{minipage}[c]{0.9\textwidth}
		\includegraphics[width=4.5in]{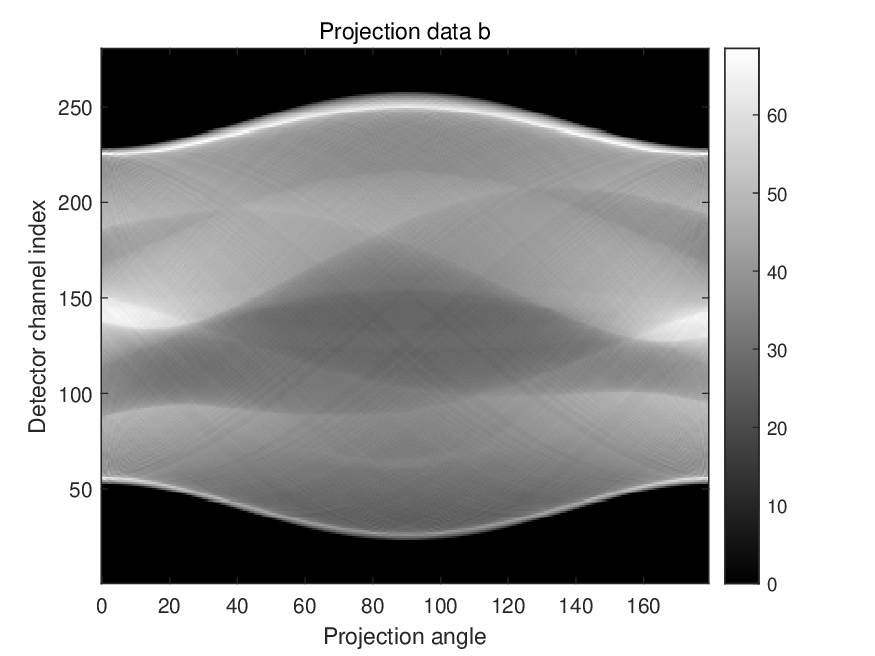}
	\end{minipage}
\end{center}
	\vspace{-1em}\caption{Pictures of the full projection data $b$ \label{Fig:5}}
\end{figure}

 \begin{figure}[htbp!]
\begin{center}
	\begin{minipage}[c]{0.9\textwidth}
		\includegraphics[width=4.5in]{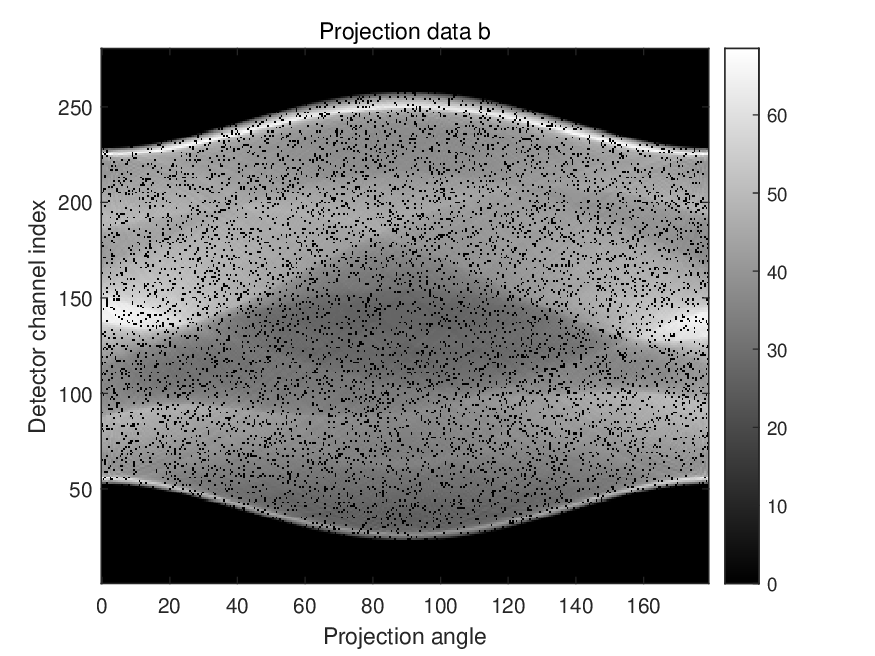}
	\end{minipage}
\end{center}
	\vspace{-1em}\caption{Pictures of the $90\%$ projection data $b$ \label{Fig:6}}
\end{figure}

 \begin{figure}[htbp!]
\begin{center}
	\begin{minipage}[c]{0.9\textwidth}
		\includegraphics[width=4.5in]{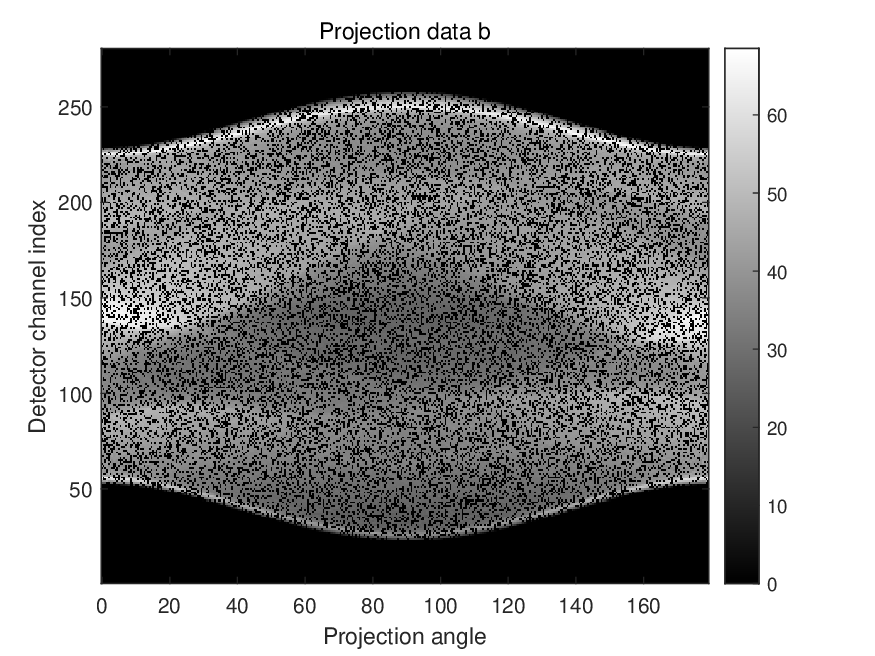}
	\end{minipage}
\end{center}
	\vspace{-1em}\caption{Pictures of the $70\%$ projection data $b$ \label{Fig:7}}
\end{figure}

Then, we generated a test image of size 128 by 128 based on the $paralleltomo$ function in the AIR Tools II package and recorded the measurements of the angles 0:0.5:179. The number of sensors per projection line is 140. The right-hand side vector is
$\mathbf{b}=\mathbf{A}\mathbf{x}_{\star}+r$, where the unique solution of the inconsistent linear systems \eqref{1-1} is obtained by reshaping the $128\times128$ Shepp-Logan medical phantom and $r$ is a nonzero vector in the null space generated by the MATLAB function $null$.

In Table \ref{Table-11}, we list the number of iteration steps, the computing times, and the relative solution errors for all tested methods. As was shown, RPLSS-GCD and RPLSS-GEK
methods show better performance than REABK, GBEK, RCD, GRCD, and PLSSRCD(10) in terms of computing time. RCD, GRCD, and RPLSS-CD (10) all failed to converge or had insufficient accuracy. RPLSS-GCD achieves extreme computational efficiency and robustness in ultra-large ill-conditioned problems, making it the best method in this test.

\begin{table}[!htbp]
    \centering
    \caption{Numerical results  of Example \ref{Example-4} for $50260\times16384$ matrix $A$   \label{Table-11}}
    \begin{tabular}{|l|c|c|c|c|}
    \hline
        \multicolumn{2}{|c|}{Algorithm} & IT & CPU &RSE \\ \hline

        \multicolumn{2}{|c|}{REABK} &  18963  & 21.8153 & 9.99e-3\\

        \multicolumn{2}{|c|}{GBEK} &  453  & 8.3650& 9.99e-3\\

        \multicolumn{2}{|c|}{RCD} &  82109  & $\dagger$ & 5.96e+2\\ 
        
        \multicolumn{2}{|c|}{GRCD}  &  77987 & $\dagger$ & 1.04e-1 \\

        \multicolumn{2}{|c|}{RPLSS-CD(10)}  & 15932& $\dagger$ & 1.57e-1 \\

        \multicolumn{2}{|c|}{RPLSS-GCD}  & 27 & \textbf{0.5096} & 9.66e-3\\

        \multicolumn{2}{|c|}{RPLSS-GEK} & 240 & 5.0662 &9.99e-3  \\ \hline

    \end{tabular}
\end{table}

\end{example}
\section{Conclusions}
\label{section-5}

In this paper, we propose a randomized projected linear system solver framework to solve Eq. \eqref{1-1}. Unlike the PLSS method that requires access to the whole coefficient matrix in each iteration, our method utilizes a random row or column selection strategy, using only partial information in each iteration. This significantly reduces memory requirements and computational cost per iteration, making the proposed method particularly suitable for large-scale, memory-constrained, or data-streaming applications. The main contributions are listed as follows.\\
\begin{itemize}
    \item For consistent systems, we proposed the RPLSS with the Gaussian Kaczmarz (RPLSS‑GK) method, which constructs each sketch column as a random linear combination of residuals and standard basis vectors. This design preserves the low memory footprint of Kaczmarz‑type iterations while inheriting the finite-termination property of PLSS. Numerical experiments demonstrate that RPLSS‑GK is remarkably robust even under row‑missing scenarios.
    \item Inspired by the double‑sequence idea of the randomized extended Kaczmarz method, we derived an extended version of RPLSS‑GK to handle inconsistent systems. This extension achieves expected linear convergence to the least‑squares solution without requiring any row or column paving strategies, making it both theoretically rigorous and practically efficient.

    \item Additionally, we devise RPLSS‑CD for inconsistent systems by selecting columns with probabilities proportional to squared norms, accumulating historical updates to overcome slow single‑component convergence. 
    \item Finally, to further balance local coordinate updates and global residual feedback, we developed the RPLSS‑GCD method. RPLSS-GCD achieves a clever trade-off by flexibly combining the current residual with randomly selected standard basis vectors, significantly accelerating the convergence speed of ill-conditioned and inconsistent problems.
\end{itemize}

\bigskip
\noindent\textbf{Contributions:} 
{All authors contributed equally to Conceptualization, Formal analysis, Investigation, Methodology, Software, Validation, Writing an original draft, Writing, Review, and Editing. All authors have read and agreed to the published version
of the manuscript.}
\vspace{2mm}

\noindent\textbf{Data Availability Statement:} {The data used to support the findings of this study are included within this article.}

\noindent\textbf{Conflict of Interest:} {The authors declare no conflict of interest.}

\appendix
	\section{Details of $\bar{\mathbf{p}}_k$}\label{App-1}    
    
\vspace{2mm}    
Consider the 
constrained optimization problem
\begin{equation}
\min_{\bar{\mathbf{p}}_k\in\mathbb{R}^n}\frac{1}{2}\|\mathbf{B\bar{\mathbf{p}}_k}\|_2^2-\mathbf{r}_{k-1}^T\mathbf{A}^T\bar{\mathbf{p}}_k
\end{equation}
\begin{equation}
\mbox{subject\ to}\quad\bar{\mathbf{S}}_k^T\mathbf{A}^T\mathbf{z}_k=0,
\end{equation}
where $\mathbf{r}_{k-1}=\mathbf{A}^T\mathbf{z}_{k-1}$, or equivalently,
\begin{equation}\label{0000}
\begin{bmatrix}
\mathbf{B}^T\mathbf{B} & \mathbf{A}\bar{\mathbf{S}}_k \\
\bar{\mathbf{S}}_k^T\mathbf{A}^T & \mathbf{0}_{k\times k}
\end{bmatrix}
\begin{bmatrix}
\bar{\mathbf{p}}_k \\
\boldsymbol{\lambda}_k
\end{bmatrix}
=
\begin{bmatrix}
\mathbf{A} \\
\bar{\mathbf{S}}_k^T
\end{bmatrix}
\left( \mathbf{A}^T\mathbf{z}_{k-1}\right).
\end{equation}
For the first term of Eq.\eqref{0000}, the update $\bar{\mathbf{p}}_k$ is of the form $$\bar{\mathbf{p}}_k=\mathbf{W}\mathbf{A}\bar{\mathbf{S}}_k(\bar{\mathbf{S}}_k^T\mathbf{A}^T\mathbf{W}\mathbf{A}\bar{\mathbf{s}}_k)^{-1}\bar{\mathbf{S}}_k^T\mathbf{A}^T\mathbf{z}_{k-1}.$$
Similar to the update ${\mathbf{p}}_k$ of PLSS in \cite{Bjj}, this  update can be rewritten as $$\bar{\mathbf{p}}_k=\bar{\mathbf{P}}_{k-1}\bar{\mathbf{g}}_{k-1}+\bar{\mathbf{\gamma}}_{k-1}\bar{\mathbf{y}}_k,$$ where $\bar{\mathbf{y}}_k=\mathbf{A}\bar{\mathbf{s}}_k$. By the
 orthogonality $\bar{\mathbf{P}}_{k-1}^T\bar{\mathbf{p}}_{k}=0$, it follows that $$\bar{\mathbf{g}}_{k-1}=-\bar{\gamma}_{k-1}\bar{\mathbf{\Theta}}_{k-1} ^{-1}\bar{\mathbf{P}}_{k-1}^T\bar{\mathbf{y}}_k,$$ where $\bar{\mathbf{\Theta}}_{k-1}=\bar{\mathbf{P}}_{k-1}^T\bar{\mathbf{P}}_{k-1}$. Following the fact $\bar{\mathbf{s}}_k^T\mathbf{A}\bar{\mathbf{p}}_k=\bar{\mathbf{s}}_k^T\mathbf{r}_{k-1}$, we have $$\bar{\gamma}_{k-1}=\frac{\bar{\mathbf{s}}_k^T\mathbf{A}^T\mathbf{z}_{k-1}}{\|\mathbf{A}\bar{\mathbf{s}}_k\|_2^2-\|\bar{\mathbf{d}}_{k-1}\|_2^2},$$ where $\bar{\mathbf{d}}_{k-1}=\bar{\mathbf{\Theta}}_{k-1} ^{-1/2}\bar{\mathbf{P}}_{k-1}^T\mathbf{A}\bar{\mathbf{s}}_k$.

\end{document}